\documentclass[11pt,reqno]{amsart}
\usepackage{a4wide}

\numberwithin{equation}{section}
\usepackage{mathrsfs}
\usepackage{amsfonts}
\usepackage{amsmath}
\usepackage{stmaryrd}
\usepackage{amssymb}
\usepackage{amsthm}
\usepackage{mathrsfs}
\usepackage{url}
\usepackage{amsfonts}
\usepackage{amscd}
\usepackage{indentfirst}
\usepackage{enumerate}
\usepackage{amsmath,amsfonts,amssymb,amsthm}
\usepackage{amsmath,amssymb,amsthm,amscd}
\usepackage{graphicx,mathrsfs}
\usepackage{appendix}
\usepackage[pagewise]{lineno}
\usepackage[numbers,sort&compress]{natbib}

\usepackage{color}
\usepackage{txfonts}
\usepackage{latexsym}
\usepackage{hyperref}
\usepackage{hypernat}
\allowdisplaybreaks

\newcommand{\D}{\mathbb{D}}
\newcommand{\R}{\mathbb{R}}
\newcommand{\Z}{\mathbb{Z}}

\newcommand{\E}{\mathbb{E}}

\renewcommand{\theequation}{\arabic{section}.\arabic{equation}}

\newtheorem{Thm}{Theorem}[section]
\newtheorem{Lem}[Thm]{Lemma}

\newtheorem{Prop}[Thm]{Proposition}

\newtheorem{Rem}[Thm]{Remark}

\begin{document}
\title[Infinitely many bubbling solutions and non-degeneracy  results]
{Infinitely many bubbling solutions and non-degeneracy  results to fractional prescribed curvature problems
}

\author{Lixiu Duan, Qing Guo}
\address[Li-Xiu Duan]{School of Mathematical Sciences, Beihang University (BUAA), Beijing 100191, P. R. China, and Key Laboratory of Mathematics, Informatics and Behavioral Semantics, Ministry of Education, Beijing 100191, P. R. China}
\email{dlx18801183265@163.com}

\address[Qing Guo]{College of Science, Minzu University of China, Beijing 100081, China}
\email{guoqing0117@163.com}

\date{\today}


\begin{abstract}
We consider the following fractional prescribed curvature problem
\begin{align}\label{eq01}
(-\Delta)^s u=  K(y)u^{2^*_s-1},\ \ u>0,\ \ y\in \R^N,
\end{align}
where $s\in(0,\frac{1}{2})$ for $N=3$, $s\in(0,1)$ for $N\geqslant4$ and $2^*_s=\frac{2N}{N-2s}$ is the fractional critical Sobolev exponent, $K(y)$ has a local maximum point in $r\in(r_0-\delta,r_0+\delta)$.

First, for any sufficient large $k$, we construct a $2k$ bubbling solution to \eqref{eq01} of some new type, which concentrate on an upper and lower surfaces of an oblate cylinder through the Lyapunov-Schmidt reduction method.  Furthermore, a non-degeneracy result of the multi-bubbling solutions is proved by use of various Pohozaev identities, which is new in the study of the fractional problems.

\end{abstract}
\maketitle {\small {\bf Keywords:}  Fractional critical elliptic equation, Infinitely many solutions, Lyapunov-Schmidt Reduction method, Non-degeneracy, Pohozaev identities.\\

{\bf 2010 AMS Subject Classifications}: 35J20, 35J60, 35B25.}
\maketitle

\section{Introduction and main results}
\setcounter{equation}{0}
In this paper, we consider the following fractional prescribed curvature equation
\begin{align}\label{eq11}
(-\Delta)^s u=  K(y)u^{2^*_s-1},\ \ u>0,\ \ y\in \R^N,
\end{align}
where $s\in(0,\frac{1}{2})$ for $N=3$, $s\in(0,1)$ for $N\geqslant4$ and $2^*_s=\frac{2N}{N-2s}$ is the fractional critical Sobolev exponent. $(-\Delta)^s$ is the nonlocal operator defined as
$$(-\Delta)^su=c(N,s)P.V.\int_{\R^N}\frac{u(x)-u(y)}{|x-y|^{N+2s}}dy,$$
where $P.V.$ is the principal value and $c(N,s)=\pi^{-(2s+\frac{N}{2})}\Gamma(s+\frac{N}{2})/\Gamma(-s)$. For more details on the fractional Laplace operator, one can refer to \cite{WCYLPM,EDNGPEV}.

Recently, problems with fractional Laplacian have been extensively studied, see for example \cite{b-c-s-s,b-c-p,b-c-p-s,c-s,c-t,f-q-t,RLFEL,RLFELLS,YJ,n-t-w,
silvestre,tan,t-x,y-y-y} and references therein. Particularly, various critical problems
for the fractional Laplacian were considered in  \cite{b-c-p,c-t,GMPY1,j-l-x} and references therein.
\medskip

When $s=1$, Wei and Yan in \cite{JS} consider the prescribed scalar curvature problem $$-\Delta u=K(|y|)u^{2^*-1}\ \ u>0,\ \ y\in \R^N,$$ and they proved the existence of infinitely many solutions concentrated on  a circle. Later, Guo, Musso, Peng and Yan in \cite{GMPY1} studied the  non-degeneracy of this $k$-bubbling solution, by which, they use the gluing method to construct a new solutions. Precisely, with $n\gg k$, they glue the $k$-bubbling solution, which
concentrates  at the vertices of the regular $k$-polygon in the $(y_1, y_2)$-plane, with a $n$-spike solution, whose centers lie in another circle in the
$(y_3, y_4)$-plane. In \cite{DLP-MM-W}, the existence of this solution $2k$ Aubin-Talenti bubbles are centred at points lying on the top and the bottom circles of a cylinder is studied, whose energy can be made arbitrarily. In this paper, we extend the results to the  fractional case nontrivially and study the corresponding non-degeneracy results.
\medskip

To illustrate our results, we first give some denotations and definitions.
Define $\dot{H}^s(\R^N)$ space as the closure of the set $C_c^\infty(\R^N)$ of compact supported smooth functions under the norm
$$\|u\|_{\dot{H}^s(\R^N)}=\|(-\Delta)^{\frac{s}{2}}u\|_{L^2(\R^N)}=\||\xi|^s\mathcal{F}(u)(\xi)\|_{L^2(\R^N)},$$
where $\mathcal{S}$ is the Schwartz space of rapidly decaying $C^\infty$ functions on $\R^N$, $\mathcal{F}$ is the Fourier transformation of $\phi$ by:
$$\mathcal{F}\phi(\xi)=\frac{1}{(2\pi)^{\frac{N}{2}}}\int_{\R^N}e^{-i\xi x}\phi(x)dx,\ \forall\phi\in\mathcal{S}.$$
Another equivalent space is defined by
$$D^s(\R^N):=\Big\{u\in L^{\frac{2N}{N-2s}}(\R^N):\|(-\Delta)^{\frac{s}{2}}u\|_{L^2(\R^N)}<\infty\Big\}.$$

Through  the extension formulation of $(-\Delta)^s$ introduced in \cite{c-s}, the equation \eqref{eq11} is equivalent to a degenerate elliptic equation with a Neumann boundary problem, which is fundamentally different from the classical elliptic equation and defined on the upper half-space $\R^N_+=\{(y,t):y\in\R^N,t>0\}$. For
$\forall$ $u\in\dot{H}^s(\R^N)$, we set
$$\tilde u(y,t)=\mathcal{P}_s[u]:=\int_{\R^N}\mathcal{P}_s(y-\xi,t)u(\xi)d\xi
,\quad (y,t)\in\R^{N+1}_+:=\R^N\times[0,+\infty),$$
where $$\mathcal{P}_s(x,t)=\beta(N,s)\frac{t^{2s}}{(|x|^2+t^2)^{\frac{N+2s}{2}}},$$
with a constant $\beta(N,s)$ such that $\int_{\R^N}\mathcal{P}_s(x,1)dx=1$. Thus, $\tilde u\in L^2(t^{1-2s},K)$ for any compact set $K$ in $\overline{\R^{N+1}_+}$, $\nabla \tilde u\in L^2(t^{1-2s},\R^{N+1}_+)$ and $\tilde u\in C^{\infty}(\R^{N+1}_+)$.

\medskip

It is well known that for any $x\in\R^N$ and $\Lambda>0$ the functions
$$U_{x,\Lambda}(y)=(4^s\gamma)^{\frac{N-2s}{4}}\Big(\frac{\Lambda}{1+\Lambda^2|y-x|^2}\Big)^{\frac{N-2s}{2}}$$
with $\gamma=\frac{\tau(\frac{N+2s}{2})}{\tau(\frac{N-2s}{2})}$,
are the only solutions to the problem (see \cite{L})
\begin{align}\label{eq12}
(-\Delta)^s u=  u^{2^*_s-1},\ \ u>0 \ \ \text{in} \ \R^N.\end{align}
Moreover,  the functions
\begin{align*}
Z_{0}(y)=\frac{\partial U}{\partial \Lambda}\Big\vert_{\Lambda=1}, \ \ \ Z_i(y)=\frac{\partial U}{\partial y_i}\Big(y\Big),\ \ i=1,2,\cdots,N
\end{align*}
solving the linearized problem
\begin{align*}
(-\Delta)^s \phi=(2^*_s-1)U^{2^*_s-2}\phi,\ \ u>0 \ \ \text{in} \ \R^N,\end{align*}
are the kernels of the linearized operator associated with \eqref{eq12}.

\medskip

Throughout this paper, we assume that $K(r)$ is a bounded radial potential function, and has a local maximum at $r_0$ satisfying
\begin{equation}\label{eq13}
K(r)=K(r_0)-c_0|r-r_0|^m+O(|r-r_0|^{m+\theta}),\ \ r\in(r_0-\delta,r_0+\delta),
\end{equation}
where $c_0>0,\delta>0$ are some constants, and
$m\in(\frac{N-2s}{2},N-2s)$. Without loss of generality, we may assume that $K(r_0)=1$.
\smallskip
For any integer $k>0$, by use of the transformation $u(y)=\mu^{-\frac{N-2s}{2}}u(\frac{|y|}{\mu})$, the equation \eqref{eq11} becomes
\begin{align}\label{eq14}
(-\Delta)^s u=  K\Big(\frac{|y|}{\mu}\Big)u^{\frac{N+2s}{N-2s}},\ \ u>0,\ \ y\in \R^N.
\end{align}

\medskip

Consider the $2k$ points
$$\overline{x}_{j}=r\Big(\sqrt{1-h^2}\cos\frac{2(j-1)\pi}{k},\sqrt{1-h^2}\sin\frac{2(j-1)\pi}{k},h,\textbf{0}\Big),\ \ \ \ j=1,\cdots,k,$$
$$\underline{x}_{j}=r\Big(\sqrt{1-h^2}\cos\frac{2(j-1)\pi}{k},\sqrt{1-h^2}\sin\frac{2(j-1)\pi}{k},-h,\textbf{0}\Big),\ \ j=1,\cdots,k,$$
where $\textbf{0}$ is the zero vector in $\R^{N-3}, r,h$ are positive parameters to be determined.
Define
\begin{align}
H_s&=\Big\{u:u\in D^{s}(\R^N),\ u\ \text{is}\ \text{even}\ \text{in}\ y_{l},\ l=2,4,5,\cdots,N,\nonumber\\
&\qquad u(\sqrt{y_{1}^2+y_{2}^2}\cos\theta',\sqrt{y_{1}^2+y_{2}^2}\sin\theta',y_{3},y'')\nonumber\\&
=u(\sqrt{y_{1}^2+y_{2}^2}\cos(\theta'+\frac{2j\pi}{k}),\sqrt{y_{1}^2+y_{2}^2}\sin(\theta'+\frac{2j\pi}{k}),y_{3},y'')\Big\},\nonumber
\end{align}
where $\theta'=\arctan\frac{y_{1}}{y_{2}}$.
We also denote the configuration space as
\begin{align}
\D=\Big\{(r,h,\Lambda):&r\in\Big[r_0\mu-\frac{1}{\mu^\theta}, r_0\mu+\frac{1}{\mu^\theta}\Big],h\in\Big[h_0-\frac{1}{\mu^\theta},h_0+\frac{1}{\mu^\theta}\Big],\nonumber\\
&\Lambda\in\Big[\Lambda_0-\frac{1}{\mu^{\frac{3}{2}\theta}}, \Lambda_0+\frac{1}{\mu^{\frac{3}{2}\theta}}\Big]\Big\}\label{eq15}
\end{align}
for\ some\  small $\theta>0$,
where $(r_0,h_0,\Lambda_0)$ is defined as in Section 3.

Define
\begin{align}\label{qe}
\mu=k^{\frac{N-2s}{N-2s-m}},h=O\Big(\frac{1}{k^\frac{N-2s-1}{N-2s+1}}\Big) \ \ \text{and} \ \ W_{r,h,\Lambda}=\sum_{i=1}^k U_{\overline{x}_{i},\Lambda}+\sum_{i=1}^k U_{\underline{x}_{i,\Lambda}}.
\end{align}
We construct $2k$ multi-bubbling  solutions concentrating at points lying on the upper
and lower circles of an oblate  cylinder.
\medskip

Precisely, our main results are as follows:

\begin{Thm}\label{th1}
Suppose that $s\in(0,1)$ for $N\geqslant 4$ and $s\in(0,\frac{1}{2})$ for $N=3$ . If $K(y)$ satisfies \eqref{eq13}, the problem $\eqref{eq14}$ has infinitely many positive solutions.
\end{Thm}
Theorem \ref{th1} is directly shown by the following result.

\begin{Thm}\label{th2}
Under the assumptions of  Theorem \ref{th1}, there exists an integer $k_0$ such that for all integer $k\geqslant k_0$, the problem  \eqref{eq14} has a solution $u_k$ of the following form:
$$
u_k=W_{r_k,h_k,\Lambda_k}(y)+\omega_k(y),
$$
where $\omega_k\in H_{s},(r_k,h_k,\Lambda_k)\in \D$ and $\omega_k$ satisfies $\|\omega_k\|_{L^\infty}\to 0$ as $k\to\infty$.

\end{Thm}

 Denote $$L_k\xi=(-\Delta)^s \xi-(2^*_s-1) K\Big(\frac{|y|}{\mu}\Big) u^{2_s^*-2}\xi.$$
  The non-degeneracy of $L_k$  makes it possible to glue two
   multiple bubbling solutions concentrated in two orthogonal subspaces together to construct infinitely many solutions of new type. More precisely, we have the following result.

\begin{Thm}\label{th3}
Assume $N\geqslant 3$, and $s\in(0,1)$, for $N\geqslant4$, $s\in(0,\frac12)$ for $N=3$. Suppose that $K(y)$ satisfies \eqref{eq13}.
Let $\xi\in H_s$ be a solution of $L_k\xi=0$. Then $\xi=0$.
\end{Thm}

\begin{Rem}
First of all, compared with the previous known results about the construction of the k-bubbling solution of the fractional equation, which concentrates on a circle, our solution obtained in this paper concentrates both on the upper and lower sides of an oblate cylinder. In this process, there is one more variable parameter, that is, the height $h$ of the cylinder, which makes the proof of the construction of the solution become more subtle.
In addition, we study the non-degeneracy of this type of bubbling solutions, which makes it possible to continue the gluing process to construct other new solutions concentrated symmetrically at arranged infinitely many points.
Moreover, in the process of using two kinds of Pohozaev identities to prove the non-degeneracy, the analysis of the main terms further shows great difference relative to the case without the parameter $h$.

Compared with the Laplacian, the nonlocal operators make it rather difficult to apply local Pohozaev identities, which is of great importance in the study of non-degeneracy and other properties of the solutions. By the aid of harmonic extension, we put two kinds of equivalent problems together, i.e., local and nonlocal ones,  and  delicately  handle  some integral terms induced by the extensions. The algebra decay at infinity of the approximate solutions, in stark contrast to the classical Laplacian, makes it crucial to give  accurate and precise estimates of the solutions.

\end{Rem}

This paper is organized as follows. In Section 2, we use the Lyapunov-Schmidt reduction procedure to get a finite dimensional setting. In section 3, the main results of the corresponding finite dimensional problems are obtained. Then in section 4, the non-degeneracy result
for the positive multi-bubbling solutions constructed in Theorem \ref{th2} is proved by use of the local Pohozaev
identities. In the Appendix, we give energy expressions and  some useful tools and estimates.

\medskip
\section{Preliminaries and the reduction framework}
Let
$$||u||_{*}=\sup_{y\in\mathbb{R}^{n}}\Big(\sum_{j=1}^k\frac{1}{(1+|y-\overline{x}_{j}|)^{\frac{N-2s}{2}+\tau}}+
\sum_{j=1}^k\frac{1}{(1+|y-\underline{x}_{j}|)^{\frac{N-2s}{2}+\tau}}\Big)^{-1}|u(y)|,$$
and
$$||f||_{**}=\sup_{y\in\mathbb{R}^{n}}\Big(\sum_{j=1}^k\frac{1}{(1+|y-\overline{x}_{j}|)^{\frac{N+2s}{2}+\tau}}+
\sum_{j=1}^k\frac{1}{(1+|y-\underline{x}_{j}|)^{\frac{N+2s}{2}+\tau}}\Big)^{-1}|f(y)|,$$
where $\frac{N-2s-m}{N-2s}<\tau<\min\{\frac{N+2s}{2}-\frac{4s}{N-2s},1+\epsilon\}$ and
$\epsilon>0$ is a small constant.
\smallskip

Denote
$$\overline{\Z}_{1,j}=\frac{\partial U_{\overline{x}_{j},\Lambda}}{\partial r}, \ \ \ \
\underline{\Z}_{1,j}=\frac{\partial U_{\underline{x}_{j},\Lambda}}{\partial r}  \ \
\text{for}\ j=1,\cdots,k,$$
$$\overline{\Z}_{2,j}=\frac{\partial U_{\overline{x}_{j},\Lambda}}{\partial h}, \ \ \ \
\underline{\Z}_{2,j}=\frac{\partial U_{\underline{x}_{j},\Lambda}}{\partial h}\ \
\text{for}\ j=1,\cdots,k, $$

$$\overline{\Z}_{3,j}=\frac{\partial U_{\overline{x}_{j},\Lambda}}{\partial \Lambda}, \ \ \ \
\underline{\Z}_{3,j}=\frac{\partial U_{\underline{x}_{j},\Lambda}}{\partial \Lambda}  \ \
\text{for}\ j=1,\cdots,k.$$

 Define 
\begin{equation}\label{e}
\E=\left\{v:v\in H_{s},\int_{\R^N} U_{\overline{x}_{j},\Lambda}^{2^*_s-2}\overline{\mathbb Z}_{l,j}v=0 \ \ \text{and}\ \int_{\R^N} U_{\underline{x}_{j},\Lambda}^{2^*_s-2}
\underline{\Z}_{l,j}v=0, \ j=1,\cdots,k,\ \ l=1,2,3\right\}.\end{equation}

Moreover, denote the splitting domains $\Omega_{j}$ for $j=1,\cdots,k$ as
$$\Omega_{j}:=\left\{ y=(y_{1},y_{2},y_{3},y'')\in\mathbb{R}^{3}\times\mathbb{R}^{N-3}:\Big\langle\frac{(y_{1},y_{2})}{|(y_{1},y_{2})|},
\Big(\cos\frac{2(j-1)\pi}{k},\sin\frac{2(j-1)\pi}{k}\Big)\Big\rangle_{\mathbb{R}^2}\geqslant\cos\frac{\pi}{k}\right\},$$
and we divide the $\Omega_{j}$ into upper and lower regions:
$$\Omega^{+}_{j}=\left\{y=(y_{1},y_{2},y_{3},y'')\in\Omega_{j},y_{3}\geqslant 0\right\},$$
$$\Omega^{-}_{j}=\left\{y=(y_{1},y_{2},y_{3},y'')\in\Omega_{j},y_{3}<0\right\}.$$

Obviously,
$$\R^N=\displaystyle\cup^k_{j=1}\Omega_j,\quad\Omega_j^+\cup\Omega_j^-$$
and
$$\Omega_j\cap\Omega_i=\varnothing,\quad
\Omega_j^+\cap\Omega_j^-=\varnothing,\quad \text{if}\ \ i\ne j.$$

Consider the following linearized problem:

\begin{equation}
\begin{cases}\label{21}
 \displaystyle(-\Delta)^s\phi_k-(2^*_s-1)K\Big(\frac{|y|}{\mu}\Big)W^{2^*_s-2}_{r,h,\Lambda}\phi_k=f_k+\sum_{i=1}^3c_i\Big(\sum_{j=1}^kU^{2^*_s-2}_{\overline{x}_{j,\Lambda}}\overline{\Z}_{i,j}+
 \sum_{j=1}^kU^{2^*_s-2}_{\underline{x}_{j,\Lambda}}\underline{\Z}_{i,j}\Big),\\
 \phi_k\in \E,
\end{cases}
\end{equation}
with some constants $c_i$.
\begin{Lem}\label{lem 21}
Assume that $\phi_k$ solves problem \eqref{21} for $f=f_k$. If  $||f_k||_{**}$ goes to zero as k goes to infinity, so does $||\phi_k||_*$.
\end{Lem}
\begin{proof}
We prove it by contradiction. Assume that there exist $f_k$ with $||f_k||_{**}\to 0$ as $k\to\infty$, $(r_k,h_k,\Lambda_k)$ satisfies \eqref{eq15} and $\phi_k$ solves problem  \eqref{21} for $f=f_k, \Lambda=\Lambda_k, r=r_k, h=h_k$ with $||\phi_k||_{*}\geqslant c>0$. Without loss of generality, we always assume that $||\phi_k||_{*}\equiv1$. For the sake of convenience, we drop the subscript $k$.

Since \eqref{21}, we get
\begin{align*}
 \displaystyle \phi(y)&\leqslant C\int_{\R^N}\frac{1}{|z-y|^{N-2s}}K\Big(\frac{|z|}{\mu}\Big)W^{2^*_s-2}_{r,h,\Lambda}\phi(z)dz+\int_{\R^N}\frac{1}{|z-y|^{N-2s}}f(z)dz\nonumber\\
 &\quad+\int_{\R^N}\frac{1}{|z-y|^{N-2s}}\sum_{i=1}^3c_i\Big(\sum_{j=1}^kU^{2^*_s-2}_{\overline{x}_{j,\Lambda}}\overline{\Z}_{i,j}+
 \sum_{j=1}^kU^{2^*_s-2}_{\underline{x}_{j,\Lambda}}\underline{\Z}_{i,j}\Big)dz.
\end{align*}

For the first term, using Lemma \ref{lemA3}, we have
\begin{align*}
\Big|&\int_{\R^N}\frac{1}{|z-y|^{N-2s}}K\Big(\frac{|z|}{\mu}\Big)W^{2^*_s-2}_{r,h,\Lambda}\phi(z)dz\Big|\nonumber\\
\leqslant &||\phi||_*\int_{\R^N}\frac{1}{|z-y|^{N-2s}}W^{2^*_s-2}_{r,h,\Lambda}\sum_{j=1}^k\Big(\frac{1}{(1+|z-\overline{x}_{j}|)^{\frac{N-2s}{2}+\tau}}+
\frac{1}{(1+|z-\underline{x}_{j}|)^{\frac{N-2s}{2}+\tau}}\Big)dz\nonumber\\
\leqslant &||\phi||_*\sum_{j=1}^k\Big(\frac{1}{(1+|y-\overline{x}_{j}|)^{\frac{N-2s}{2}+\tau+\theta}}+
\frac{1}{(1+|y-\underline{x}_{j}|)^{\frac{N-2s}{2}+\tau+\theta}}\Big),
\end{align*}
and for the second term, from Lemma \ref{lemA2}, it holds that
\begin{align*}
\Big|&\int_{\R^N}\frac{1}{|z-y|^{N-2s}}f(z)dz\Big|\nonumber\\
\leqslant &||f||_{**}\int_{\R^N}\frac{1}{|z-y|^{N-2s}}\sum_{j=1}^k\Big(\frac{1}{(1+|z-\overline{x}_{j}|)^{\frac{N+2s}{2}+\tau}}+
\frac{1}{(1+|z-\underline{x}_{j}|)^{\frac{N+2s}{2}+\tau}}\Big)dz\nonumber\\
\leqslant &||f||_{**}\sum_{j=1}^k\Big(\frac{1}{(1+|y-\overline{x}_{j}|)^{\frac{N-2s}{2}+\tau}}+
\frac{1}{(1+|y-\underline{x}_{j}|)^{\frac{N-2s}{2}+\tau}}\Big).
\end{align*}

Since
\begin{align}\label{22}
|\overline{\Z}_{1,j}|\leqslant\frac{C}{(1+|y-\overline{x}_{j}|)^{N-2s}},\ |\overline{\Z}_{2,j}|\leqslant\frac{Cr}{(1+|y-\overline{x}_{j}|)^{N-2s}},\
|\overline{\Z}_{3,j}|\leqslant\frac{C}{(1+|y-\overline{x}_{j}|)^{N-2s}},\nonumber\\
|\underline{\Z}_{1,j}|\leqslant\frac{C}{(1+|y-\underline{x}_{j}|)^{N-2s}},\ |\underline{\Z}_{2,j}|\leqslant\frac{Cr}{(1+|y-\underline{x}_{j}|)^{N-2s}},\
|\underline{\Z}_{3,j}|\leqslant\frac{C}{(1+|y-\underline{x}_{j}|)^{N-2s}}.
\end{align}
Combing \eqref{22}, we  obtain,
\begin{align*}
&\quad\Big|\int_{\R^N}\frac{1}{|z-y|^{N-2s}}\Big(\sum_{j=1}^kU^{2^*_s-2}_{\overline{x}_{j,\Lambda}}\overline{\Z}_{i,j}+
 \sum_{j=1}^kU^{2^*_s-2}_{\underline{x}_{j,\Lambda}}\underline{\Z}_{i,j}\Big)dz\Big|\nonumber\\
&\leqslant C\sum_{j=1}^k\Big(\frac{1+r\pmb{\delta_{i2}}}{(1+|y-\overline{x}_{j}|)^{\frac{N-2s}{2}+\tau}}+
\frac{1+r\pmb{\delta_{i2}}}{(1+|y-\underline{x}_{j}|)^{\frac{N-2s}{2}+\tau}}\Big),
\end{align*}
where
\begin{equation*}
\pmb{\delta_{ij}}=
\begin{cases}
1,\ \ \ &\text{if} \ i=j,\\
0,\ \ \ &\text{if} \ i\neq j.
\end{cases}
\end{equation*}
Next we estimate $c_i, i=1,2,3$. Multiplying \eqref{21} by $\overline{\Z}_{i,1}$ and integrating, it shows that $c_i$ satisfies
\begin{align}\label{23}
&\quad\displaystyle\Big\langle(-\Delta)^s\phi-(2^*_s-1)K\Big(\frac{|y|}{\mu}\Big)W^{2^*_s-2}_{r,h,\Lambda}\phi,\overline{\Z}_{i,1}\Big\rangle-\Big\langle t,\overline{\Z}_{i,1}\Big\rangle\nonumber\\
&=\Big\langle\sum_{i=1}^3c_i\Big(\sum_{j=1}^kU^{2^*_s-2}_{\overline{x}_{j,\Lambda}}\overline{\Z}_{i,j}+
 \sum_{j=1}^kU^{2^*_s-2}_{\underline{x}_{j,\Lambda}}\underline{\Z}_{i,j}\Big),\overline{\Z}_{i,1}\Big\rangle.
\end{align}
where $\langle x,y\rangle=\int_{\R^N}xydz$. By using Lemma \ref{lemA1}, we have
\begin{align}\label{24}
\Big|\Big\langle f,\overline{\Z}_{i,1}\Big\rangle\Big|
&\leqslant ||f||_{**}\int_{\R^N}\frac{1+r\pmb{\delta_{i2}}}{(1+|z-\overline{x}_{1}|)^{N-2s}}\sum_{j=1}^k\Big(\frac{1}{(1+|y-\overline{x}_{j}|)^{\frac{N+2s}{2}+\tau}}+
\frac{1}{(1+|y-\underline{x}_{j}|)^{\frac{N+2s}{2}+\tau}}\Big)dz\nonumber\\
&\leqslant C(1+r\pmb{\delta_{i2}})||f||_{**}.
\end{align}

Next, we calculate
\begin{align*}
&\quad \Big\langle(-\Delta)^s\phi-(2^*_s-1)K\Big(\frac{|y|}{\mu}\Big)W^{2^*_s-2}_{r,h,\Lambda}\phi,\overline{\Z}_{i,1}\Big\rangle\nonumber\\
&=\Big\langle(-\Delta)^s\overline{\Z}_{i,1}-(2^*_s-1)K\Big(\frac{|y|}{\mu}\Big)W^{2^*_s-2}_{r,h,\Lambda}\overline{\Z}_{i,1},\phi\Big\rangle\nonumber\\
&=(2^*_s-1)\Big\langle(U^{2^*_s-2}_{\overline{x}_{1,\Lambda}}-W^{2^*_s-2}_{r,h,\Lambda})\overline{\Z}_{i,1},
\phi\Big\rangle-(2^*_s-1)\Big\langle\Big[K\Big(\frac{|y|}{\mu}\Big)-1\Big]W^{2^*_s-2}_{r,h,\Lambda}\overline{\Z}_{i,1},\phi\Big\rangle.\nonumber\\
&:=H_1-H_2.
\end{align*}
For  $H_1$, it holds that
\begin{align}\label{25}
H_1&\leqslant C\|\phi\|_*\int_{\R^N}(U^{2^*_s-2}_{\overline{x}_{1,\Lambda}}-W^{2^*_s-2}_{r,h,\Lambda})\frac{1+r\pmb{\delta_{i2}}}{(1+|z-\overline{x}_{1}|)^{N-2s}}\nonumber\\
&\quad\times\sum_{j=1}^k\Big(\frac{1}{(1+|y-\overline{x}_{j}|)^{\frac{N-2s}{2}+\tau}}+
\frac{1}{(1+|y-\underline{x}_{j}|)^{\frac{N-2s}{2}+\tau}}\Big)dz\nonumber\\
&\leqslant C\|\phi\|_*\frac{1+r\pmb{\delta_{i2}}}{\mu^\sigma}.
\end{align}
Using the condition of $K(|y|)$, we  obtain $H_2$,
\begin{align}\label{26}
H_2&\leqslant C\|\phi\|_*\int_{\R^N}\Big[K\Big(\frac{|y|}{\mu}\Big)-1\Big]W^{2^*_s-2}_{r,h,\Lambda}\overline{\Z}_{i,1}
\sum_{j=1}^k\Big(\frac{1}{(1+|y-\overline{x}_{j}|)^{\frac{N-2s}{2}+\tau}}+
\frac{1}{(1+|y-\underline{x}_{j}|)^{\frac{N-2s}{2}+\tau}}\Big)dz\nonumber\\
&\leqslant C\|\phi\|_*\int_{\{\Big||y|-\mu r_0\Big|\leqslant\sqrt{\mu}\}}\Big[K\Big(\frac{|y|}{\mu}\Big)-1\Big]W^{2^*_s-2}_{r,h,\Lambda}\overline{\Z}_{i,1}
\sum_{j=1}^k\Big(\frac{1}{(1+|y-\overline{x}_{j}|)^{\frac{N-2s}{2}+\tau}}+
\frac{1}{(1+|y-\underline{x}_{j}|)^{\frac{N-2s}{2}+\tau}}\Big)dz\nonumber\\
&\quad+ C\|\phi\|_*\int_{\{\Big||y|-\mu r_0\Big|>\sqrt{\mu}\}}\Big[K\Big(\frac{|y|}{\mu}\Big)-1\Big]W^{2^*_s-2}_{r,h,\Lambda}\overline{\Z}_{i,1}
\sum_{j=1}^k\Big(\frac{1}{(1+|y-\overline{x}_{j}|)^{\frac{N-2s}{2}+\tau}}+
\frac{1}{(1+|y-\underline{x}_{j}|)^{\frac{N-2s}{2}+\tau}}\Big)dz\nonumber\\
&\leqslant C\frac{\|\phi\|_*}{\mu^{\frac{m}{2}}}\int_{\{\Big||y|-\mu r_0\Big|\leqslant\sqrt{\mu}\}}\Big[K\Big(\frac{|y|}{\mu}\Big)-1\Big]W^{2^*_s-2}_{r,h,\Lambda}\frac{1+r\pmb{\delta_{i2}}}{(1+|z-\overline{x}_{1}|)^{N-2s}}
\sum_{j=1}^k\frac{1}{(1+|y-\overline{x}_{j}|)^{\frac{N-2s}{2}+\tau}}dz\nonumber\\
&\quad+ C\frac{\|\phi\|_*}{\mu^{\sigma}}\int_{\{\Big||y|-\mu r_0\Big|>\sqrt{\mu}\}}\Big[K\Big(\frac{|y|}{\mu}\Big)-1\Big]W^{2^*_s-2}_{r,h,\Lambda}\frac{1+r\pmb{\delta_{i2}}}{(1+|z-\overline{x}_{1}|)^{N-2s}}
\sum_{j=1}^k\frac{1}{(1+|y-\overline{x}_{j}|)^{\frac{N-2s}{2}+\tau-2\sigma}})dz\nonumber\\
&\leqslant C\|\phi\|_*\frac{1+r\pmb{\delta_{i2}}}{\mu^\sigma}.
\end{align}
By using \eqref{25} and \eqref{26}, we   get
\begin{align}\label{27}
\Big\langle(-\Delta)^s\phi-(2^*_s-1)K\Big(\frac{|y|}{\mu}\Big)W^{2^*_s-2}_{r,h,\Lambda}\phi,\overline{\Z}_{i,1}\Big\rangle=C\|\phi\|_*\frac{1+r\pmb{\delta_{i2}}}{\mu^\sigma}.
\end{align}
However, there is a constant $c>0$ satisfies
\begin{align}\label{28}
\Big\langle\Big(\sum_{j=1}^kU^{2^*_s-2}_{\overline{x}_{j,\Lambda}}\overline{\Z}_{t,j}+
 \sum_{j=1}^kU^{2^*_s-2}_{\underline{x}_{j,\Lambda}}\underline{\Z}_{t,j}\Big),\overline{\Z}_{i,1}\Big\rangle=\overline{c}_t\delta_{ti}(1+\pmb{\delta_{i2}}r^2),
 \end{align}
where $\overline{c}_t$ is a constant.
So, we  find that by substituting \eqref{24}, \eqref{27} and \eqref{28} into equation \eqref{23},
$$c_i=\frac{1+r\pmb{\delta_{i2}}}{1+r^2\pmb{\delta_{i2}}}O\Big(\frac{1}{\mu^{\sigma}}||\phi||_*+||f||_{**}\Big)=o(1).$$

Thus,
\begin{align}\label{29}
\displaystyle||\phi||_*\leqslant\Big(||f||_{**}+\frac{\displaystyle\sum_{j=1}^k\Big(\frac{1}{(1+|y-\overline{x}_{j}|)^{\frac{N-2s}{2}+\tau+\theta}}+
\frac{1}{(1+|y-\underline{x}_{j}|)^{\frac{N-2s}{2}+\tau+\theta}}\Big)}{\displaystyle\sum_{j=1}^k\Big(\frac{1}{(1+|y-\overline{x}_{j}|)^{\frac{N-2s}{2}+\tau}}+
\frac{1}{(1+|y-\underline{x}_{j}|)^{\frac{N-2s}{2}+\tau}}\Big)}\Big).
\end{align}
Since $||\phi||_*=1$, we find that there is $R>0$ from \eqref{29}, such that
\begin{align}\label{210}
\displaystyle||\phi||_{L^\infty(B_{R}(\overline{x}_{i}))}\geqslant a>0,
\end{align}
for some $\overline{x}_{i}$. But $\overline{\phi}(y)=\phi(y-\overline{x}_{i})$ converges uniformly in any compact set to a solution $u$ of
\begin{align}\label{211}
(-\Delta)^s u-(2^*_s-1)U_{0,\Lambda}^{2^*_s-2}u=0,\ \text{in} \ \R^N,
\end{align}
for some $\Lambda\in[L_1,L_2]$, $u$ must be a linear combination of the functions
$$\frac{\partial U_{0,\Lambda}}{\partial \Lambda}\Big\arrowvert_{\Lambda=1},\ \ \frac{\partial U_{0,1}}{\partial y_1},\ \ \ \frac{\partial U_{0,1}}{\partial y_3}.$$

However $u$ is perpendicular to the kernel of \eqref{211}. So, $u=0$, which is in conflict with \eqref{210}.

\end{proof}

From Lemma \ref{lem 21}, using the same argument as the proof for Proposition 2.2 in \cite{DLP-MM-W}, we  prove the following result:
\begin{Lem}\label{lem 22}
There exist $k_0>0$ and a constant $C>0$, independent of k, such that for all $k\geqslant k_0$ and all $f\in L^{\infty}(\R^N)$, problem \eqref{21} has a unique solution $\phi\equiv \mathbf{L}_k(f)$. Furthermore,
\begin{align}\label{212}
\|\mathbf{L}_k(f)\|_*\leqslant C||f||_{**},\ \ |c_i|\leqslant \frac{C}{1+r\pmb{\delta_{i2}}}||f||_{**},\quad i=1,2,3.
\end{align}
\end{Lem}
\medskip

Now, we consider the following problem \eqref{213},
\begin{equation}
\begin{cases}\label{213}
 \displaystyle(-\Delta)^s(W_{r,h,\Lambda}+\phi)=K\Big(\frac{|y|}{\mu}\Big)(W_{r,h,\Lambda}+\phi)^{2^*_s-1}+\sum_{i=1}^3c_i\Big(\sum_{j=1}^kU^{2^*_s-2}_{\overline{x}_{j,\Lambda}}\overline{\Z}_{i,j}+
 \sum_{j=1}^kU^{2^*_s-2}_{\underline{x}_{j,\Lambda}}\underline{\Z}_{i,j}\Big),\\
 \phi\in \E.
\end{cases}
\end{equation}
where $\E$ is as \eqref{e}.

Rewrite above equation \eqref{213} as
\begin{equation}
\begin{cases}\label{214}
 \displaystyle(-\Delta)^s\phi-(2^*_s-1)K\Big(\frac{|y|}{\mu}\Big)W_{r,h,\Lambda}^{2^*_s-2}\phi=N(\phi)+l_k+\sum_{i=1}^3c_i\Big(\sum_{j=1}^kU^{2^*_s-2}_{\overline{x}_{j,\Lambda}}\overline{\Z}_{i,j}+
 \sum_{j=1}^kU^{2^*_s-2}_{\underline{x}_{j,\Lambda}}\underline{\Z}_{i,j}\Big),\\
 \phi\in\E.
\end{cases}
\end{equation}
where
$$ N(\phi)=K\Big(\frac{|y|}{\mu}\Big)\Big((W_{r,h,\Lambda}+\phi)^{2^*_s-1}-W_{r,h,\Lambda}^{2^*_s-1}-(2^*_s-1)W_{r,h,\Lambda}^{2^*_s-2}\phi\Big),$$
and
\begin{equation}\label{215}
l_k=K\Big(\frac{|y|}{\mu}\Big)W_{r,h,\Lambda}^{2^*_s-1}-\Big(\sum_{j=1}^kU^{2^*_s-1}_{\overline{x}_{j,\Lambda}}+
 \sum_{j=1}^kU^{2^*_s-1}_{\underline{x}_{j,\Lambda}}\Big).
\end{equation}

In order to prove that \eqref{214} is uniquely solvable  using the contraction mapping Theorem, we need to estimate $N(\phi)$ and $l_k$.

 \begin{Lem}\label{lem 23}
If $s\in(0,1)$ for $N\geqslant 4$ and $s\in(0,\frac{1}{2})$ for $N=3$, and $||\phi||_*\leqslant 1$, then
$$||N(\phi)||_{**}\leqslant C||\phi||_*^{\min(2^*_s-1,2)}.$$
\end{Lem}
\begin{proof}
Since the proof is similar to the \cite{YJ}, we omit it here.
\end{proof}
Next, we estimate $l_k$.
\begin{Lem}\label{lem 24}
If $s\in(0,1)$ for $N\geqslant 4$ and $s\in(0,\frac{1}{2})$ for $N=3$, then there is a small $\epsilon>0$, such that
$$||l_k||_{**}\leqslant C\Big(\frac{1}{\mu}\Big)^{\frac{m}{2}+\epsilon}.$$
\end{Lem}
\begin{proof}
First we  rewrite equation \eqref{215} as
\begin{align*}
l_k&=K\Big(\frac{|y|}{\mu}\Big)\Big(W_{r,h,\Lambda}^{2^*_s-1}-\sum_{j=1}^kU^{2^*_s-1}_{\overline{x}_{j,\Lambda}}-
\sum_{j=1}^kU^{2^*_s-1}_{\underline{x}_{j,\Lambda}}\Big)+\sum_{j=1}^k\Big(U^{2^*_s-1}_{\overline{x}_{j,\Lambda}}+
 U^{2^*_s-1}_{\underline{x}_{j,\Lambda}}\Big)\Big[K\Big(\frac{|y|}{\mu}\Big)-1\Big]\nonumber\\
 &:=J_1+J_2.
 \end{align*}
By symmetry, supposing $y\in\Omega_{1}^+$, we obtain
 $$|y-\overline{x}_{1}|\leqslant|y-\underline{x}_{1}|\ \ \text{and}\ \ |y-\overline{x}_{1}|\leqslant|y-\overline{x}_{j}|\leqslant|y-\underline{x}_{j}|,j=2,3,\cdots,k.$$
So, it holds that
 \begin{align*}
 |J_1|&=K\Big(\frac{|y|}{\mu}\Big)\Big|\Big(\Big(\sum_{j=1}^kU_{\overline{x}_{j},\Lambda}+\sum_{j=1}^kU_{\underline{x}_{j},\Lambda}\Big)
 ^{2^*_s-1}-\sum_{j=1}^kU^{2^*_s-1}_{\overline{x}_{j,\Lambda}}-
\sum_{j=1}^kU^{2^*_s-1}_{\underline{x}_{j,\Lambda}}\Big)\Big|\nonumber\\
&\leqslant C\Big(U_{\overline{x}_{1},\Lambda}^{2^*_s-2}\Big(\sum_{j=2}^kU_{\overline{x}_{j},\Lambda}+\sum_{j=1}^kU_{\underline{x}_{j},\Lambda}\Big)
 +\Big(\sum_{j=2}^kU_{\overline{x}_{j},\Lambda}+
\sum_{j=1}^kU_{\underline{x}_{j},\Lambda}\Big)^{2^*_s-1}\Big)\nonumber\\
&\leqslant C\frac{1}{(1+|y-\overline{x}_{1}|)^{4s}}\Big(\sum_{j=2}^k\frac{1}{(1+|y-\overline{x}_{j}|)^{N-2s}}+\sum_{j=1}^k
\frac{1}{(1+|y-\underline{x}_{j}|)^{N-2s}}\Big)\nonumber\\
&\quad+\Big(\sum_{j=2}^k\frac{1}{(1+|y-\overline{x}_{j}|)^{N-2s}}\Big)^{2^*_s-1}
\end{align*}

By Lemma \ref{lemA1}, since $\displaystyle \Big(\frac{N+2s}{2}-\tau\Big)\frac{m}{N-2s}>\frac{m}{2}+\epsilon$ , from \eqref{qe} we  get
\begin{align} \label{217}
\frac{1}{(1+|y-\overline{x}_{1}|)^{4s}}\sum_{j=2}^k\frac{1}{(1+|y-\overline{x}_{j}|)^{N-2s}}&\leqslant\frac{1}{(1+|y-\overline{x}_{1}|)^{\frac{N+2s}{2}+\tau}}
\sum_{j=2}^k\frac{1}{|\overline{x}_{j}-\overline{x}_{1}|^{\frac{N+2s}{2}-\tau}}\nonumber\\
&\leqslant\frac{1}{(1+|y-\overline{x}_{1}|)^{\frac{N+2s}{2}+\tau}}
\Big(\frac{k}{\mu\sqrt{1-h^2}}\Big)^{\frac{N+2s}{2}-\tau}\nonumber\\
&\leqslant\frac{1}{(1+|y-\overline{x}_{1}|)^{\frac{N+2s}{2}+\tau}}
\Big(\frac{1}{\mu}\Big)^{\frac{m}{2}+\epsilon}
\end{align}

and
\begin{align} \label{218}
\frac{1}{(1+|y-\overline{x}_{1}|)^{4s}}\sum_{j=1}^k\frac{1}{(1+|y-\underline{x}_{j}|)^{N-2s}}&\leqslant\frac{1}{(1+|y-\overline{x}_{1}|)^{\frac{N+2s}{2}+\tau}}
\sum_{j=1}^k\frac{1}{|\underline{x}_{j}-\overline{x}_{1}|^{\frac{N+2s}{2}-\tau}}\nonumber\\
&\leqslant\frac{1}{(1+|y-\overline{x}_{1}|)^{\frac{N+2s}{2}+\tau}}
\Big(\frac{1}{rh}\Big)^{\frac{N+2s}{2}-\tau}\frac{hk}{\sqrt{1-h^2}}\nonumber\\
&\leqslant\frac{1}{(1+|y-\overline{x}_{1}|)^{\frac{N+2s}{2}+\tau}}
\Big(\frac{1}{\mu}\Big)^{\frac{m}{2}+\epsilon}.
\end{align}

By H\"older inequalities, for $\frac{N+2s}{4s}(\frac{N-2s}{2}-\frac{N-2s}{N+2s}\tau)>1$, there holds that
\begin{align}\label{219}
\Big(\sum_{j=2}^k\frac{1}{(1+|y-\overline{x}_{j}|)^{N-2s}}\Big)^{2^*_s-1}
&\leqslant\sum_{j=2}^k\frac{1}{(1+|y-\overline{x}_{j}|)^{\frac{N+2s}{2}+\tau}}\Big(\sum_{j=2}^k\frac{1}{(|\overline{x}_{j}-\overline{x}_{1}|)^{\frac{N+2s}{4s}(\frac{N-2s}{2}-\frac{N-2s}{N+2s}\tau)}}
\Big)^{\frac{4s}{N-2s}}\nonumber\\
&\leqslant\frac{1}{(1+|y-\overline{x}_{1}|)^{\frac{N+2s}{2}+\tau}}
\Big(\frac{1}{\mu}\Big)^{\frac{m}{2}+\epsilon}.
\end{align}

Combining \eqref{217},\eqref{218},\eqref{219}, we obtain

 $$||J_1||_{**}\leqslant
\Big(\frac{1}{\mu}\Big)^{\frac{m}{2}+\epsilon}.$$

Next, we consider $J_2$ to estimate
\begin{align*}
&\qquad J_2\leqslant2\Big(\Big[K\Big(\frac{|y|}{\mu}\Big)-1\Big]U^{2^*_s-1}_{\overline{x}_{1,\Lambda}}+
 \Big[K\Big(\frac{|y|}{\mu}\Big)-1\Big]\sum_{j=2}^kU^{2^*_s-1}_{\overline{x}_{j,\Lambda}}\Big).
\end{align*}
For the first term, for every $y\in\Omega_{1}^+$, when $\arrowvert|y|-\mu r_0\arrowvert\geqslant\delta\mu$, we obtain
$$|y-\overline{x}_{1}|\geqslant||y|-|\overline{x}_{1}||\geqslant||y|-\mu r_0|-||\overline{x}_{1}|-\mu r_0||\geqslant\frac{1}{2}\delta\mu.$$
Hence,
\begin{align}\label{220}
[K\Big(\frac{|y|}{\mu}\Big)-1]U_{\overline{x}_{1},\Lambda}^{2^{*}_s-1}
&\leqslant C\frac{1}{(1+|y-\overline{x}_{1}|)^{\frac{N+2s}{2}+\tau}}\frac{1}{\mu^{\frac{N+2s}{2}-\tau}}\nonumber\\
&\leqslant\frac{1}{(1+|y-\overline{x}_{1}|)^{\frac{N+2s}{2}+\tau}}
\Big(\frac{1}{\mu}\Big)^{\frac{m}{2}+\epsilon},
\end{align}
where the last inequalities is due to $m\in(\frac{N-2s}{2},N-2s)$, $\frac{m}{2}+\epsilon<\frac{N+2s}{2}-\tau$.

\smallskip

When $y\in\Omega_{1}^+$ and $\arrowvert|y|-\mu r_0\arrowvert<\delta\mu$, we have $|y-\overline{x}_{1}|\leqslant||y|-\mu r_0|+||\overline{x}_{1}|-\mu r_0||\leqslant2\delta\mu$. From
the condition of $K(|y|)$, we get
\begin{align*}
\Big[K\Big(\frac{|y|}{\mu}\Big)-1\Big] \leqslant\frac{C}{\mu^m}\Big(||y|-|\overline{x}_{1}||^m+\frac{1}{\mu^{m\theta}}\Big).
\end{align*}
Thus,
\begin{align}\label{221}
\Big[K\Big(\frac{|y|}{\mu}\Big)-1\Big]U_{\overline{x}_{1},\Lambda}^{2^{*}_s-1}
&\leqslant \frac{C}{(1+|y-\overline{x}_{1}|)^{\frac{N+2s}{2}+\sigma}}\frac{1}{\mu^{\frac{m}{2}+\epsilon}}.
\end{align}
For the second term, we  get
\begin{align}\label{222}
\Big[K\Big(\frac{|y|}{\mu}\Big)-1\Big]\sum_{j=2}^kU^{2^*_s-1}_{\overline{x}_{j,\Lambda}}&\leqslant\frac{C}{(1+|y-\overline{x}_{1}|)^{\frac{N+2s}{2}}}\sum_{j=2}^k\frac{C}{(1+|y-\overline{x}_{j}|)^{\frac{N+2s}{2}}}\nonumber\\
&\leqslant\frac{C}{(1+|y-\overline{x}_{1}|)^{\frac{N+2s}{2}+\sigma}}\sum_{j=2}^k\frac{C}{|\overline{x}_{1}-\overline{x}_{j}|^{\frac{N+2s}{2}-\sigma}}\nonumber\\
&\leqslant\frac{C}{(1+|y-\overline{x}_{1}|)^{\frac{N+2s}{2}+\sigma}}\Big(\frac{1}{\mu}\Big)^{\frac{m}{2}+\epsilon}.
\end{align}
Combing \eqref{220},\eqref{221},\eqref{222}, we  obtain
 $$||J_2||_{**}\leqslant C\Big(\frac{1}{\mu}\Big)^{\frac{m}{2}+\epsilon}.$$
\end{proof}

Next, we  prove the following result through the contraction mapping Theorem.
\begin{Lem}\label{lem 25}
There exist $k_0>0$, such that for all $k\geqslant k_0, (r,h,\Lambda)\in\D$. The problem  \eqref{213} has a unique solution $\phi=\phi(r,h,\Lambda)$, satisfying
\begin{equation}\label{223}
||\phi||_*\leqslant C\Big(\frac{1}{\mu}\Big)^{\frac{m}{2}+\epsilon}
\end{equation}
and
\begin{equation}\label{224}
|c_t|\leqslant \frac{C}{1+\pmb{\delta_{i2}}r}\Big(\frac{1}{\mu}\Big)^{\frac{m}{2}+\epsilon}
,\end{equation}
where $\epsilon>0$ is a small constant.
\end{Lem}
\begin{proof}
Define
$$\bar{S}:=\left\{\phi:\phi\in \E\cap C^\infty_0(\R^N), \|\phi\|_{*}\leqslant C\Big(\frac{1}{\mu}\Big)^{\frac{m}{2}+\epsilon}
\right\}.$$
By Lemma \ref{lem 22}, $\mathbf{L}$ is invertible, so \eqref{214} is equivalent to
$$\phi=A(\phi):=\mathbf{L}_k(N(\phi))+\mathbf{L}_k(l_k).$$

We  will conclude  that $A$ is a contraction mapping  from $\bar{S}$ to $\bar{S}.$ In fact, by Lemmas \ref{lem 23} and Lemma \ref{lem 24}, we get
\begin{align*}
\|A(\phi)\|_*&\leqslant C\|N(\phi)\|_{**}+C\|l_k\|_{**}\nonumber\\
&\leqslant\Big(\frac{1}{\mu}\Big)^{\frac{m}{2}+\epsilon}.
\end{align*}

So, $A$ maps $\bar{S}$ to $\bar{S}$. And if $p\leqslant3$, we have $|N'(t)|\leqslant|t|^{p-2}$; if $p>3$, we have $|N'(t)|\leqslant C(W^{2^*_s-3}_{r,h,\Lambda}|t|+|t|^{p-2})$. Thus if $p\leqslant3$, for all $\phi_1, \phi_2\in\bar{S}$, it holds that
\begin{align*}
\|A(\phi_{1})-A(\phi_{2})\|_*&=\|L_k(N(\phi_1))-L_k(N(\phi_2))\|_*\\
&\leqslant\frac{1}{2}\|\phi_1-\phi_2\|_*.
\end{align*}

Therefore, $A$ is a contracting mapping. The case $p>3$  can be obtained similarly. According to the contraction mapping Theorem there is a unique $\phi\in \E$, such that
$$\phi=A(\phi).$$

Besides, by  Lemma \ref{lem 22}, we  conclude \eqref{223} and \eqref{224}. \end{proof}

\section{Finite dimensional problem}
Define
$$F(r,h,\Lambda)=I(W_{r,h,\Lambda}+\phi),\ \ \ \forall(r,h,\Lambda)\in
\D,$$
where $r=|\overline{x}_1|$ and $\phi$ is the function gained in Lemma \ref{lem 25}.

In this section, we will give the energy expansions for $F(r,h,\Lambda), \frac{\partial F(r,h,\Lambda)}{\partial\Lambda}$ and $\frac{\partial F(r,h,\Lambda)}{\partial h}$.
\begin{Prop}\label{pp31}
We have
\begin{align}\label{31}
F(r,h,\Lambda)&=k\Big(A+\frac{A_1}{\Lambda^m \mu^m}+\frac{A_2}{\Lambda^{m-2}\mu^m}(\mu r_{0}-r)^2
-\frac{B_{1}k^{N-2s}}{\Lambda^{N-2s}(r\sqrt{1-h^2})^{N-2s}}
-\frac{B_2}{\Lambda^{N-2s}(rh)^{N-2s}}\frac{hk}{\sqrt{1-h^2}}\nonumber\\
&\quad+\frac{C}{\mu^m}|\mu r_0-r|^{2+\theta}+\frac{C}{\mu^{m+\theta}}
+O\Big(\frac{1}{\mu^{N-2s-\frac{N-2s-m}{N-2s}-\frac{(N-2s-m)(N-2s-1)^2}{(N-2s)(N-2s+1)}+\theta}}\Big)\Big),
\end{align}
where $ A, A_1, A_2, B_{1}\ and\ B_{2}$ are some positive constants.
\end{Prop}

\begin{proof}
Since
$$\langle I'(W_{r,h,\Lambda}+\phi),\phi\rangle=0,\quad \quad \forall\phi\in\bar{S},$$
there is a constant $\theta\in(0,1)$ such that
\begin{align*}
F(r,h,\Lambda)&=I(W_{r,h,\Lambda})+\frac{1}{2}D^2I(W_{r,h,\Lambda}+\theta\phi)(\phi,\phi)\nonumber\\
&=I(W_{r,h,\Lambda})-\frac{2^*_s-1}{2}\int_{\R^N}K\Big(\frac{|y|}{\mu}\Big)\Big((W_{r,h,\Lambda}+\theta\phi)^{2^*_s-2}-W_{r,h,\Lambda}^{2^*_s-2}\Big)\phi^2
+\frac{1}{2}\int_{\R^N}(N(\phi)+l_k)\phi\nonumber\\
&=I(W_{r,h,\Lambda})+O\Big(\int_{\R^N}(|\phi|^{2^*_s}+|N(\phi)||\phi|+|l_k||\phi|)\Big).
\end{align*}
Since
\begin{align*}
\int_{\R^N}(|N(\phi)||\phi|+|l_k||\phi|)&\leqslant C(||N(\phi)||_{**}+||l_k||_{**})||\phi||_{*}\nonumber\\
&\quad\times\int_{\R^N}\Big(\sum_{j=1}^k\frac{1}{(1+|y-\overline{x}_{j}|)^{\frac{N-2s}{2}+\tau}}+
\sum_{j=1}^k\frac{1}{(1+|y-\underline{x}_{j}|)^{\frac{N-2s}{2}+\tau}}\Big)\nonumber\\
&\quad\times
\Big(\sum_{i=1}^k\frac{1}{(1+|y-\overline{x}_{i}|)^{\frac{N+2s}{2}+\tau}}+
\sum_{i=1}^k\frac{1}{(1+|y-\underline{x}_{i}|)^{\frac{N+2s}{2}+\tau}}\Big),
\end{align*}
then by Lemma \ref{lemA1}, we obtain
\begin{align*}
&\sum_{j=1}^k\frac{1}{(1+|y-\overline{x}_{j}|)^{\frac{N-2s}{2}+\tau}}\sum_{i=1}^k\frac{1}{(1+|y-\overline{x}_{i}|)^{\frac{N+2s}{2}+\tau}}
\leqslant C\sum_{j=1}^k\frac{1}{(1+|y-\overline{x}_{j}|)^{N+\tau}}.
\end{align*}
Thus
\begin{align*}
\int_{\R^N}(|N(\phi)||\phi|+|l_k||\phi|)&\leqslant Ck(||N(\phi)||_{**}+||l_k||_{**})||\phi||_{*}\nonumber\\&\leqslant
\frac{Ck}{\mu^{m+\theta}}.
\end{align*}

On the other hand, we have
$$\int_{\R^N}|\phi|^{2^*_s}\leqslant C||\phi||_*^{2^*_s}\int_{\R^N}\Big(\sum_{j=1}^k\frac{1}{(1+|y-\overline{x}_{j}|)^{\frac{N-2s}{2}+\tau}}+
\sum_{j=1}^k\frac{1}{(1+|y-\underline{x}_{j}|)^{\frac{N-2s}{2}+\tau}}\Big)^{2^*_s}.$$
By using Lemma \ref{lemA1}, if $y\in\Omega_{1}^+$, we can find a small constant $\alpha>0$ such that
\begin{align*}
&\quad\sum_{j=2}^k\frac{1}{(1+|y-\overline{x}_{j}|)^{\frac{N-2s}{2}+\tau}}+
\sum_{j=1}^k\frac{1}{(1+|y-\underline{x}_{j}|)^{\frac{N-2s}{2}+\tau}}\nonumber\\
&\leqslant C\sum_{j=1}^k\frac{1}{(1+|y-\overline{x}_{1}|)^{\frac{N-2s}{2}+\frac{1}{2}\alpha}}\Big(\frac{1}{|\overline{x}_{j}-\overline{x}_{1}|^{\tau-\frac{1}{2}\alpha}}
+\frac{1}{|\underline{x}_{j}-\overline{x}_{1}|^{\tau-\frac{1}{2}\alpha}}\Big)\nonumber\\
&\leqslant C\sum_{j=1}^k\frac{1}{(1+|y-\overline{x}_{1}|)^{\frac{N-2s}{2}+\frac{1}{2}\alpha}},
\end{align*}
and so
$$\Big(\sum_{j=1}^k\frac{1}{(1+|y-\overline{x}_{j}|)^{\frac{N-2s}{2}+\tau}}+
\sum_{j=1}^k\frac{1}{(1+|y-\underline{x}_{j}|)^{\frac{N-2s}{2}+\tau}}\Big)^{2^*_s}
\leqslant \frac{C}{(1+|y-\overline{x}_{1}|)^{N+2^*_s\frac{1}{2}\alpha}}.$$
Since
$$\int_{\R^N}\Big(\sum_{j=1}^k\frac{1}{(1+|y-\overline{x}_{j}|)^{\frac{N-2s}{2}+\tau}}+
\sum_{j=1}^k\frac{1}{(1+|y-\underline{x}_{j}|)^{\frac{N-2s}{2}+\tau}}\Big)^{2^*_s}\leqslant Ck,$$
we have
\begin{align*}
\int_{\R^N}|\phi|^{2^*_s}\leqslant Ck||\phi||_*^{2^*_s}\leqslant \frac{Ck}{\mu^{m+\theta}}.
\end{align*}
\end{proof}

\begin{Prop}\label{pp33}
\begin{align}\label{33}
\frac{\partial F(r,h,\Lambda)}{\partial h}&=k\Big(\frac{B_2(N-2s-1)k}{\Lambda^{N-2s}r^{N-2s}h^{N-2s}\sqrt{1-h^2}}
-\frac{B_1(N-2s)k^{N-2s}h}{\Lambda^{N-2s}r^{N-2s}(\sqrt{1-h^2})^{N-2s+2}}\\
&\quad+O\Big(\frac{1}{\mu^{N-2s-\frac{N-2s-m}{N-2s}-\frac{(N-2s-m)(N-2s-1)}{(N-2s)(N-2s+1)}+\theta}}\Big)\Big).
\end{align}
where $  B_{1}\ and\ B_{2}$ are some positive constants.
\end{Prop}
\begin{proof}We have
\begin{align*}
\frac{\partial F(r,h,\Lambda)}{\partial h}&=\Big\langle I'(W_{r,h,\Lambda}+\phi),\frac{\partial W_{r,h,\Lambda}}{\partial h}+\frac{\partial\phi}{\partial h}\Big\rangle\nonumber\\
&=\Big\langle I'(W_{r,h,\Lambda}+\phi),\frac{\partial W_{r,h,\Lambda}}{\partial h}\Big\rangle
+\sum_{t=1}^3\sum_{j=1}^k c_t\Big\langle\Big(U^{2^*_s-2}_{\overline{x}_{j,\Lambda}}\overline{\Z}_{t,j}+
 U^{2^*_s-2}_{\underline{x}_{j,\Lambda}}\underline{\Z}_{t,j}\Big),\frac{\partial\phi}{\partial h}\Big\rangle\nonumber\\
 &=\frac{\partial I(W_r,h,\Lambda)}{\partial h}+\Big\langle K\Big(\frac{|y|}{\mu}\Big)[(W_{r,h,\Lambda}+\phi)^{2^*_s-1}-W_{r,h,\Lambda}^{2^*_s-1}],\frac{\partial W_{r,h,\Lambda}}{\partial h}\Big\rangle
\nonumber\\&\quad+\sum_{t=1}^3\sum_{j=1}^k c_t\Big\langle\Big(U^{2^*_s-2}_{\overline{x}_{j,\Lambda}}\overline{\Z}_{t,j}+
 U^{2^*_s-2}_{\underline{x}_{j,\Lambda}}\underline{\Z}_{t,j}\Big),\frac{\partial\phi}{\partial h}\Big\rangle.
\end{align*}
In views of the orthogonality, we get
$$\Big\langle\Big(U^{2^*_s-2}_{\overline{x}_{j,\Lambda}}\overline{\Z}_{t,j}+
 U^{2^*_s-2}_{\underline{x}_{j,\Lambda}}\underline{\Z}_{t,j}\Big),\frac{\partial\phi}{\partial h}\Big\rangle=
 -\Big\langle\frac{\partial}{\partial h}\Big(U^{2^*_s-2}_{\overline{x}_{j,\Lambda}}\overline{\Z}_{t,j}+
 U^{2^*_s-2}_{\underline{x}_{j,\Lambda}}\underline{\Z}_{t,j}\Big),\phi\Big\rangle.$$
Through Lemma \ref{lem 25}, it holds that
\begin{align*}
&\quad\Big|\sum_{j=1}^k c_t\Big\langle\Big(U^{2^*_s-2}_{\overline{x}_{j,\Lambda}}\overline{\Z}_{t,j}+
 U^{2^*_s-2}_{\underline{x}_{j,\Lambda}}\underline{\Z}_{t,j}\Big),\frac{\partial\phi}{\partial h}\Big\rangle\Big|\nonumber\\
 &\leqslant C|c_t|||\phi||_*\int_{\R^N}\Big(\frac{\partial U^{2^*_s-2}_{\overline{x}_{j,\Lambda}}\overline{\Z}_{t,j}}{\partial h}+\frac{\partial U^{2^*_s-2}_{\underline{x}_{j,\Lambda}}\underline{\Z}_{t,j}}{\partial h}\Big)\nonumber\\
&\quad\times\Big(\sum_{j=1}^k\frac{1}{(1+|y-\overline{x}_{j}|)^{\frac{N-2s}{2}+\tau}}+
\sum_{j=1}^k\frac{1}{(1+|y-\underline{x}_{j}|)^{\frac{N-2s}{2}+\tau}}\Big)\nonumber\\
&\leqslant C|c_t|||\phi||_*\int_{\R^N}\frac{\mu(1+r\pmb{\delta_{i2}})}{(1+|y-\overline{x}_{j}|)^{N+2s+1}}\nonumber\\
&\quad\times\Big(\sum_{j=1}^k\frac{1}{(1+|y-\overline{x}_{j}|)^{\frac{N-2s}{2}+\tau}}+
\sum_{j=1}^k\frac{1}{(1+|y-\underline{x}_{j}|)^{\frac{N-2s}{2}+\tau}}\Big)\\
&\leqslant C\Big(\frac{1}{\mu}\Big)^{m-1+2\epsilon}.
\end{align*}
We have
\begin{align*}
&\quad\int_{\R^N}K\Big(\frac{|y|}{\mu}\Big)(W_{r,h,\Lambda}+\phi)^{2^*_s-1}\frac{\partial W_{r,h,\Lambda}}{\partial h}\nonumber\\
&=\int_{\R^N}K\Big(\frac{|y|}{\mu}\Big)W_{r,h,\Lambda}^{2^*_s-1}\frac{\partial W_{r,h,\Lambda}}{\partial h}
+(2^*_s-1)\int_{\R^N}K\Big(\frac{|y|}{\mu}\Big)W_{r,h,\Lambda}^{2^*_s-2}\frac{\partial W_{r,h,\Lambda}}{\partial h}\phi
+O\Big(\int_{\R^N}|\phi|^{2^*_s}\Big).
\end{align*}
Furthermore, for every $\phi\in\E$, we obtain
\begin{align*}
&\quad\int_{\R^N}K\Big(\frac{|y|}{\mu}\Big)W_{r,h,\Lambda}^{2^*_s-2}\frac{\partial W_{r,h,\Lambda}}{\partial h}\phi\\
&=\int_{\R^N}K\Big(\frac{|y|}{\mu}\Big)\Big(W_{r,h,\Lambda}^{2^*_s-2}\frac{\partial W_{r,h,\Lambda}}{\partial h}-
\sum_{j=1}^kU^{2^*_s-2}_{\overline{x}_{j,\Lambda}}\frac{\partial U_{\overline{x}_{j},\Lambda}}{\partial h}
-\sum_{j=1}^kU^{2^*_s-2}_{\underline{x}_{j,\Lambda}}\frac{\partial U_{\underline{x}_{j},\Lambda}}{\partial h}\Big)\phi\\
&\quad+\int_{\R^N}\Big(K\Big(\frac{|y|}{\mu}\Big)-1\Big)\Big(\sum_{j=1}^kU^{2^*_s-2}_{\overline{x}_{j,\Lambda}}\frac{\partial U_{\overline{x}_{j},\Lambda}}{\partial h}
+\sum_{j=1}^kU^{2^*_s-2}_{\underline{x}_{j,\Lambda}}\frac{\partial U_{\underline{x}_{j},\Lambda}}{\partial h}\Big)\phi\nonumber\\
&=2k\int_{\Omega_{1}^{+}}K\Big(\frac{|y|}{\mu}\Big)\Big(W_{r,h,\Lambda}^{2^*_s-2}\frac{\partial W_{r,h,\Lambda}}{\partial h}-
\sum_{j=1}^kU^{2^*_s-2}_{\overline{x}_{j,\Lambda}}\frac{\partial U_{\overline{x}_{j},\Lambda}}{\partial h}
-\sum_{j=1}^kU^{2^*_s-2}_{\underline{x}_{j,\Lambda}}\frac{\partial U_{\underline{x}_{j},\Lambda}}{\partial h}\Big)\phi\nonumber\\
&\quad+2k\int_{\R^N}\Big(K\Big(\frac{|y|}{\mu}\Big)-1\Big)\Big(U^{2^*_s-2}_{\overline{x}_{1},\Lambda}\frac{\partial U_{\overline{x}_{1},\Lambda}}{\partial h}
+U^{2^*_s-2}_{\underline{x}_{1},\Lambda}\frac{\partial U_{\underline{x}_{1},\Lambda}}{\partial h}\Big)\phi.
\end{align*}

Since
\begin{align*}
\Big|\int_{\Omega_{1}^{+}}K\Big(\frac{|y|}{\mu}\Big)\Big(W_{r,h,\Lambda}^{2^*_s-2}\frac{\partial W_{r,h,\Lambda}}{\partial h}-
\sum_{j=1}^kU^{2^*_s-2}_{\overline{x}_{j,\Lambda}}\frac{\partial U_{\overline{x}_{j},\Lambda}}{\partial h}
-\sum_{j=1}^kU^{2^*_s-2}_{\underline{x}_{j,\Lambda}}\frac{\partial U_{\underline{x}_{j},\Lambda}}{\partial h}\Big)\phi\Big|\leqslant C\Big(\frac{1}{\mu}\Big)^{m-1+2\epsilon}
\end{align*}
and
\begin{align*}
&\quad\Big|\int_{\R^N}\Big(K\Big(\frac{|y|}{\mu}\Big)-1\Big)\Big(U^{2^*_s-2}_{\overline{x}_{1},\Lambda}\frac{\partial U_{\overline{x}_{1},\Lambda}}{\partial h}
+U^{2^*_s-2}_{\underline{x}_{1},\Lambda}\frac{\partial U_{\underline{x}_{1},\Lambda}}{\partial h}\Big)\phi\Big|\nonumber\\
&\leqslant
\Big|\int_{||y|-\mu r_0|\leqslant \sqrt{\mu}}\Big(K\Big(\frac{|y|}{\mu}\Big)-1\Big)\Big(U^{2^*_s-2}_{\overline{x}_{1},\Lambda}\frac{\partial U_{\overline{x}_{1},\Lambda}}{\partial h}
+U^{2^*_s-2}_{\underline{x}_{1},\Lambda}\frac{\partial U_{\underline{x}_{1},\Lambda}}{\partial h}\Big)\phi\Big|\nonumber\\
&\quad+
\Big|\int_{||y|-\mu r_0|\geqslant \sqrt{\mu}}\Big(K\Big(\frac{|y|}{\mu}\Big)-1\Big)\Big(U^{2^*_s-2}_{\overline{x}_{1},\Lambda}\frac{\partial U_{\overline{x}_{1},\Lambda}}{\partial h}
+U^{2^*_s-2}_{\underline{x}_{1},\Lambda}\frac{\partial U_{\underline{x}_{1},\Lambda}}{\partial h}\Big)\phi\Big|
\nonumber\\
&\leqslant C\Big(\frac{1}{\mu}\Big)^{m-1+2\epsilon},
\end{align*}
then by Lemma \ref{lemA7}, it is easy to show that
$$\Big(\frac{1}{\mu}\Big)^{m-1+2\epsilon} \leqslant \frac{1}{\mu^{N-2s-\frac{N-2s-m}{N-2s}-\frac{(N-2s-m)(N-2s-1)}{(N-2s)(N-2s+1)}+\theta}}.$$
\end{proof}

\begin{Prop}\label{pp32}
\begin{align}\label{32}
\frac{\partial F(r,h,\Lambda)}{\partial\Lambda}&=k\Big(-\frac{mA_1}{\Lambda^{m+1} \mu^m}-\frac{A_2(m-2)}{\Lambda^{m-1}\mu^m}(\mu r_{0}-r)^2
+\frac{B_{1}(N-2s)k^{N-2s}}{\Lambda^{N-2s+1}(r\sqrt{1-h^2})^{N-2s}}\nonumber\\
&\quad+\frac{B_2(N-2s)}{\Lambda^{N-2s+1}(rh)^{N-2s}}\frac{hk}{\sqrt{1-h^2}}
+O\Big(\frac{1}{\mu^{N-2s-\frac{N-2s-m}{N-2s}-\frac{(N-2s-m)(N-2s-1)^2}{(N-2s)(N-2s+1)}+\theta}}\Big)\Big),
\end{align}
where $  A_1,  B_{1}\ and\ B_{2}$ are some positive constants.
\end{Prop}
\begin{proof}
The proof of this Lemma is similar to Proposition \ref{pp33}, and we omit it here.
\end{proof}

Now we  analyze $F(r,h,\Lambda),\frac{\partial F(r,h,\Lambda)}{\partial  h}\ \text{and} \ \frac{\partial F(r,h,\Lambda)}{\partial  \Lambda}$.

Let $h_0$ be the solution of

$$
\frac{B_2(N-2s-1)k}{h^{N-2s}\sqrt{1-h^2}}
-\frac{B_1(N-2s)k^{N-2s}h}{(\sqrt{1-h^2})^{N-2s+2}}=0.
$$
It holds that
$$h_0=\frac{1}{\Big(1+\Big(\frac{B_1(N-2s)k^{N-2s-1}}{B_2}\Big)^{\frac{2}{N-2s+1}}\Big)^{\frac{1}{2}}}=O\Big(\frac{1}{k^{\frac{N-2s-1}{N-2s+1}}}\Big).$$
Let $\Lambda_0$ be the solution of
$$-\frac{mA_1}{\Lambda^{m+1} \mu^m}
+\frac{B_{1}(N-2s)k^{N-2s}}{\Lambda^{N-2s+1}\Big(r\sqrt{1-h_0^2}\Big)^{N-2s}}
+\frac{B_2(N-2s)}{\Lambda^{N-2s+1}(rh_0)^{N-2s}}\frac{hk}{\sqrt{1-h_0^2}}=0.$$
We get
$$\Lambda_0=\frac{1}{\mu(mA_1)^{\frac{1}{N-2s-m}}}\Big(\frac{B_1(N-2s)k^{N-2s}}{\Big(\sqrt{1-h_0^2}\Big)^{N-2s}}+\frac{B_2(N-2s)k}{h_0^{N-2s-1}\sqrt{1-h_0^2}}\Big)^{\frac{1}{N-2s-m}}.$$
Since
$$\frac{r}{\mu}=r_0+O(\frac{1}{\mu^{1+\theta}}),\ \ h=h_0+o(\frac{1}{\mu^\theta})$$
and for some positive constants $B_3, B_4, B_5$ and $B_6$, by \eqref{31} we get
\begin{align}\label{b6}
&\quad\frac{B_{1}k^{N-2s}}{(r\sqrt{1-h^2})^{N-2s}}+\frac{B_2}{(rh)^{N-2s}}\frac{hk}{\sqrt{1-h^2}}\\
&=\frac{B_{3}}{\mu^{m}}+\frac{B_4}{\mu^{N-2s-\frac{N-2s-m}{N-2s}-\frac{(N-2s-m)(N-2s-1)^2}{(N-2s)(N-2s+1)}}}
+\frac{B_5}{\mu^{m+\epsilon}}+\frac{B_6}{\mu^{N-2s-\frac{N-2s-m}{N-2s}-\frac{(N-2s-m)(N-2s-1)^2}{(N-2s)(N-2s+1)}}}\Big(1-\frac{h}{h_0}\Big)^2\nonumber\\
&\quad+O\Big(\frac{1}{\mu^{N-2s-\frac{N-2s-m}{N-2s}-\frac{(N-2s-m)(N-2s-1)^2}{(N-2s)(N-2s+1)}}}\Big(1-\frac{h}{h_0}\Big)^3\Big).\nonumber
\end{align}

\medskip

Rewrite \eqref{31}, \eqref{32} and \eqref{33} respectively as
\begin{align}\label{34}
F(r,h,\Lambda)&=k\Big(A+\frac{A_1}{\Lambda^m \mu^m}+\frac{A_2}{\Lambda^{m-2}\mu^m}(\mu r_{0}-r)^2
-\frac{B_{3}}{\Lambda^{N-2s}\mu^{m}}+\frac{C}{\mu^m}|\mu r_0-r|^{2+\theta}
\nonumber\\
&\quad
-\frac{B_4}{\Lambda^{N-2s}\mu^{N-2s-\frac{N-2s-m}{N-2s}-\frac{(N-2s-m)(N-2s-1)^2}{(N-2s)(N-2s+1)}}}
+\frac{B_6}{\Lambda^{N-2s}\mu^{N-2s-\frac{N-2s-m}{N-2s}-\frac{(N-2s-m)(N-2s-1)^2}{(N-2s)(N-2s+1)}}}\Big(1-\frac{h}{h_0}\Big)^2\nonumber\\
&\quad+O\Big(\frac{1}{\mu^{N-2s-\frac{N-2s-m}{N-2s}-\frac{(N-2s-m)(N-2s-1)^2}{(N-2s)(N-2s+1)}}}\Big(1-\frac{h}{h_0}\Big)^3
+\frac{1}{\mu^{N-2s-\frac{N-2s-m}{N-2s}-\frac{(N-2s-m)(N-2s-1)^2}{(N-2s)(N-2s+1)}+\theta}}\Big)\Big),
\end{align}

\begin{align}\label{35}
\frac{\partial F(r,h,\Lambda)}{\partial\Lambda}&=k\Big(-\frac{mA_1}{\Lambda^{m+1} \mu^m}+\frac{A_2(m-2)}{\Lambda^{m-1}\mu^m}\Big(\mu r_{0}-r\Big)^2
+\frac{B_{3}(N-2s)}{\Lambda^{N-2s+1}\mu^{m}}
+O\Big(\frac{1}{\mu^m}|\mu r_0-r|^{2+\theta}\Big)\Big),
\end{align}
and
\smallskip
\begin{align}\label{36}
\frac{\partial F(r,h,\Lambda)}{\partial h}&=\frac{k}{\Lambda^{N-2s}}\Big(\frac{B_7}{\mu^{N-2s-\frac{N-2s-m}{N-2s}-\frac{(N-2s)(N-2s-m)}{N-2s+1}}}\Big(1-\frac{h}{h_0}\Big)+
O\Big(\frac{1}{\mu^{N-2s-\frac{N-2s-m}{N-2s}-\frac{(N-2s)(N-2s-m)}{N-2s+1}}}\Big(1-\frac{h}{h_0}\Big)^2\Big)
\Big)\nonumber\\
&\quad+kO\Big(\frac{1}{\mu^{N-2s-\frac{N-2s-m}{N-2s}-\frac{(N-2s-m)(N-2s-1)}{(N-2s)(N-2s+1)}+\theta}}\Big),
\end{align}
where $A, A_1, A_2, B_3, B_4, B_5, B_6,$ and $B_7$ are positive constants defined in Lemma \ref{lemA5} and \eqref{b6}.
\smallskip

Define
$$\bar{F}(r,h,\Lambda)=-F(r,h,\Lambda),\quad\quad (r,h,\Lambda)\in \D.$$

Let
$$\alpha_2=k(-A+\beta),$$
and
$$\alpha_1=k\Big(-A-\Big(\frac{A_1}{\Lambda^{m}_0}-\frac{B_{3}}{\Lambda^{N-2s}_0}\Big)
\frac{1}{\mu^m}-\frac{1}{\mu^{m+\frac{5}{2}\theta}}\Big),$$
where $\theta,\beta>0$ is a small constant.\\
\smallskip

Denote the energy level set
$$\bar{F}^\alpha=\{(r,h,\Lambda)|(r,h,\Lambda)\in\D,\bar{F}
(r,h,\Lambda)\leqslant \alpha\}.$$
Consider
\begin{equation*}
\begin{cases}
\frac{dr}{dt}=-D_r\bar{F},t>0;\\
\frac{dh}{dt}=-D_h\bar{F},t>0;\\
\frac{d\Lambda}{dt}=-D_\Lambda\bar{F},t>0;\\
(r,h,\Lambda)\in \bar{F}^{\alpha_2}.
\end{cases}
\end{equation*}

\begin{Prop}\label{34}
The flow $(r(t),h(t),\Lambda(t))$ does not leave $\D$ before it reaches $\bar{F}^{\alpha_{1}}$.
\begin{proof}
If $\Lambda=\Lambda_0+\frac{1}{\mu^{\frac{3}{2}\theta}}$, such that $|r-\mu r_0|\leqslant\frac{1}{\mu^\theta}$, and $|h-h_0|\leqslant\frac{1}{\mu^{\theta}}$, we  gain from  \eqref{35} that
$$\frac{\partial \bar{F}(W_{r,h,\Lambda})}{\partial\Lambda}=k\Big(\frac{c}{\mu^{m+\frac{3}{2}\theta}}+O\Big(\frac{1}{\mu^{m+2\theta}}\Big)\Big)>0,$$
where $c$ is a positive constant. Thus, the flow does not leave $\D$.\\

Similarly, if $\Lambda=\Lambda_0-\frac{1}{\mu^{\frac{3}{2}\theta}}$, such that $|r-\mu r_0|\leqslant\frac{1}{\mu^\theta}$, and $|h-h_0|\leqslant\frac{1}{\mu^{\theta}}$, we  gain from  \eqref{35} that
$$\frac{\partial \bar{F}(W_{r,h,\Lambda})}{\partial\Lambda}=k\Big(-\frac{c}{\mu^{m+\frac{3}{2}\theta}}+O\Big(\frac{1}{\mu^{m+2\theta}}\Big)\Big)<0.$$
where $c$ is a positive constant. Thus, the flow does not leave $\D$.\\

If $h=h_0+\frac{1}{\mu^{\theta}}$, such that $|r-\mu r_0|\leqslant\frac{1}{\mu^\theta}$, and $|\Lambda-\Lambda_0|\leqslant\frac{1}{\mu^{\frac{3}{2}\theta}}$, we  gain from \eqref{36}  that
$$\frac{\partial \bar{F}(W_{r,h,\Lambda})}{\partial h}=-\frac{k}{\Lambda^{N-2s}}\Big(\frac{B_7}{\mu^{N-2s-\frac{N-2s-m}{N-2s}-\frac{(N-2s)(N-2s-m)}{N-2s+1}+\theta}}+
O\Big(\frac{1}{\mu^{N-2s-\frac{N-2s-m}{N-2s}-\frac{(N-2s)(N-2s-m)}{N-2s+1}+2\theta}}
\Big)\Big)<0.$$
Thus, the flow does not leave $\D$.\\

Similarly, if $h=h_0-\frac{1}{\mu^{\theta}}$, such that $|r-\mu r_0|\leqslant\frac{1}{\mu^\theta}$, and $|\Lambda-\Lambda_0|\leqslant\frac{1}{\mu^{\frac{3}{2}\theta}}$, we  gain from  \eqref{36} that
$$\frac{\partial \bar{F}(W_{r,h,\Lambda})}{\partial h}=\frac{k}{\Lambda^{N-2s}}\Big(\frac{B_7}{\mu^{N-2s-\frac{N-2s-m}{N-2s}-\frac{(N-2s)(N-2s-m)}{N-2s+1}+\theta}}+
O\Big(\frac{1}{\mu^{N-2s-\frac{N-2s-m}{N-2s}-\frac{(N-2s)(N-2s-m)}{N-2s+1}+2\theta}}
\Big)\Big)>0.$$
Thus, the flow does not leave $\D$.\\

Assuming $|r-\mu r_0|=\frac{1}{\mu^\theta}$, since $|\Lambda-\Lambda_0|\leqslant\frac{1}{\mu^{\frac{3}{2}\theta}},|h-h_0|\leqslant\frac{1}{\mu^{\theta}}$,
we get
\begin{align}\label{37}
\frac{B_{3}}{\Lambda^{N-2s}}-\frac{A_1}{\Lambda^{m}}&=\Big(\frac{B_{3}}{\Lambda^{N-2s}_0}-\frac{A_1}{\Lambda^{m}_0}\Big)+O\Big(|\Lambda-\Lambda_0|^2\Big)\nonumber\\
&=\Big(\frac{B_{3}}{\Lambda^{N-2s}_0}-\frac{A_1}{\Lambda^{m}_0}\Big)+O\Big(\frac{1}{\mu^{3\theta}}\Big).
\end{align}
So, using \eqref{34}, \eqref{37}, we obtain
\begin{align}\label{38}
\bar{F}(r,h,\Lambda)=k\Big(-A+\Big(\frac{B_{3}}{\Lambda^{N-2s}_0}-\frac{A_1}{\Lambda^{m}_0}\Big)\frac{1}{\mu^m}-\frac{A_2}{\Lambda^{m-2}\mu^{m+2\theta}}
+O\Big(\frac{1}{\mu^{m+\frac{5}{2}\theta}}\Big)\Big)<\alpha_1.
\end{align}
\end{proof}
\end{Prop}

\begin{proof}[\textbf{Proof of Theorem \ref{th2}}:]

Define
\begin{align*}
\Gamma=\Big\{&f:f(r,h,\Lambda)=\Big(f_1(r,h,\Lambda),f_2(r,h,\Lambda),f_3(r,h,\Lambda)\Big)
\in \D,(r,h,\Lambda)\in \D,\nonumber\\
&f(r,h,\Lambda)=(r,h,\Lambda),
\text{if}\  |r-\mu r_0|=\frac{1}{\mu^\theta}\Big\}.
\end{align*}

Let
$$c=\mathop{\inf}\limits_{f\in\tau}\mathop{\max}\limits_{(r,h,\Lambda)\in \D}\bar{F}(f(r,h,\Lambda)).$$
In fact we can follow \cite{JS} to prove\\

$(i)\alpha_1<c<\alpha_2;$\\

$(ii)\mathop{\sup}\limits_{|r-\mu r_0|=\frac{1}{\mu^\theta}}\bar{F}(f(r,h,\Lambda))<\alpha_1,\forall f\in\Gamma.$

 Thus $c$ is a critical value of $\bar{F}$.

\end{proof}

\section{The non-degeneracy of the solutions}
In this section, we prove the non-degeneracy of the 2k-bubbling solutions constructed by Theorem \ref{th2}. Recall
\begin{equation}\label{eq41}
(-\Delta)^s u= K\Big(\frac{|y|}{\mu}\Big) u^{2_s^*-1},
\end{equation}

\begin{equation*}
(-\Delta)^s \xi=(2^*_s-1) K\Big(\frac{|y|}{\mu}\Big) u^{2_s^*-2}\xi,
\end{equation*}
 and the solution $u_k$ for \eqref{eq41}, satisfying $$\displaystyle u_k=W_{r,h,\Lambda}+\omega_k,$$
where $\displaystyle W_{r,h,\Lambda}=\sum_{i=1}^k U_{\overline{x}_{i},\Lambda}+\sum_{i=1}^k U_{\underline{x}_{i,\Lambda}}.$

In order to apply local Pohozaev identities, we quote the extension of $\tilde u$ and $\tilde\xi$ to have
\begin{equation*}
\begin{cases}
\displaystyle div(t^{1-2s}\nabla\tilde u)=0,  &{x\in\R^{N+1}_+},\\
\displaystyle-\lim_{t\rightarrow0}t^{1-2s}\partial_t\tilde u(y,t)= K\Big(\frac{|y|}{\mu}\Big)u^{2^*_s-1}   &{x\in\R^{N}},
\end{cases}
\end{equation*}
and
\begin{equation*}
\begin{cases}
\displaystyle div(t^{1-2s}\nabla\tilde \xi)=0,  &{x\in\R^{N+1}_+},\\
\displaystyle-\lim_{t\rightarrow0}t^{1-2s}\partial_t\tilde \xi(y,t)= (2^*_s-1) K\Big(\frac{|y|}{\mu}\Big) u^{2_s^*-2}\xi  &{x\in\R^{N}}.
\end{cases}
\end{equation*}
For any smooth domain $\Omega$ in $\R^N$,
we set
\begin{align*}
\Omega^+=\{(y,t),y\in\Omega,t>0\}\subseteq\R^{N+1}_+,\\
\partial''\Omega^+=\{(y,t),y\in\partial\Omega,t>0\}\subseteq\R^{N+1}_+.
\end{align*}
There hold the following  Pohozaev identities.

\begin{Lem}\label{lem41}
There hold that
\begin{align}\label{42}
&\quad-\int_{\partial''\Omega^+}t^{1-2s}\frac{\partial\tilde u}{\partial\nu}\frac{\partial\tilde \xi}{\partial y_i}
-\int_{\partial''\Omega^+}t^{1-2s}\frac{\partial\tilde \xi}{\partial\nu}\frac{\partial\tilde u}{\partial y_i}
+ \int_{\partial''\Omega^+}t^{1-2s}\langle\nabla\tilde u,\nabla\tilde\xi\rangle\nu_i\nonumber\\
&=\int_{\Omega}\frac{\partial K\Big(\frac{|y|}{\mu}\Big)}{\partial y_i}u^{2^*_s-1}\xi+\int_{\partial \Omega} K\Big(\frac{|y|}{\mu}\Big)u^{2^*_s-1}\xi\nu_i,
\end{align}
and
\begin{align}\label{43}
&\quad\int_{\Omega}u^{2^*_s-1}\xi\langle\nabla K\Big(\frac{|y|}{\mu}\Big),y-x_0\rangle\\
&=\int_{\partial \Omega}K\Big(\frac{|y|}{\mu}\Big)u^{2^*_s-1}\xi\langle\nu,y-x_0\rangle
+\int_{\partial''\Omega^+}t^{1-2s}\frac{\partial\tilde u}{\partial\nu}\langle\nabla\tilde\xi,Y-X_0\rangle
+\int_{\partial''\Omega^+}t^{1-2s}\frac{\partial\tilde \xi}{\partial\nu}\langle\nabla\tilde u,Y-X_0\rangle\nonumber\\&
\quad-\int_{\partial''\Omega^+}t^{1-2s}\langle\nabla\tilde u,\nabla\tilde\xi\rangle\langle\nu,Y-X_0\rangle
+\frac{N-2s}2\int_{\partial''\Omega^+}t^{1-2s}\tilde\xi\frac{\partial\tilde u}{\partial\nu}
+\frac{N-2s}2\int_{\partial''\Omega^+}t^{1-2s}\tilde u\frac{\partial\tilde\xi}{\partial\nu},\nonumber
\end{align}
where $Y=(y,t),X_0=(x_0,0)\in\Omega^+$.
\end{Lem}

\begin{proof}

For a similar proof, we can refer to \cite{GNNT}.
\end{proof}

\smallskip
\smallskip
\smallskip
To obtain the non-degeneracy result, we should improve the estimates of the $2k$-bubbling solution of \eqref{eq41} obtained in Theorem \ref{th2}.
Precisely, we have the following result.

\begin{Lem}\label{lem42}
There holds that
\begin{align*}
|u_k(y)|\leqslant C\sum_{j=1}^k\frac{1}{(1+|y-\overline{x}_{j}|)^{N-2s}}
+C\sum_{j=1}^k\frac{1}{(1+|y-\underline{x}_{j}|)^{N-2s}},\ \ \text{for}\ \text{all}\ y\in\R^N.
\end{align*}
\end{Lem}

\begin{proof}
Since \eqref{eq41}, for some constant $\sigma_{N,s}$,  we have
\begin{align*}
u_k=\sigma_{N,s}\int_{\R^N}\frac{1}{|z-y|^{N-2s}}K(\frac {|z|}{\mu}) u^{2^*_s-1}_k(z)dz.
\end{align*}
We estimate by Lemma \ref{lemA3} that
\begin{align*}
|u_k|&\leqslant C\int_{\R^N}\frac{1}{|z-y|^{N-2s}} u^{2^*_s-1}_k(z)dz\\
&\leqslant C\int_{\R^N}\frac{1}{|y-z|^{N-2s}}\Big(\sum_{j=1}^k\frac{1}{(1+|z-\overline{x}_{j}|)^{\frac{N-2s}{2}+\tau}}
+\sum_{j=1}^k\frac{1}{(1+|z-\underline{x}_{j}|)^{\frac{N-2s}{2}+\tau}}\Big)^{2^*_s-1}dz\\
&\leqslant C\int_{\R^N}\frac{1}{|z-y|^{N-2s}}\sum_{j=1}^k\frac{1}{(1+ |z-\overline{x}_{j}|)^{\frac{N+2s}{2}+\tau_1+\frac{(N+2s)(\tau-\tau_1)}{N-2s}}}
\Big(\sum_{j=2}^k\frac{1}{(|\overline{x}_{1}-\overline{x}_{j}|)^{\tau_1}}\Big)^{2^*_s-2}dz\\
&\quad+C\int_{\R^N}\frac{1}{|z-y|^{N-2s}}\sum_{j=1}^k\frac{1}{(1+ |z-\overline{x}_{j}|)^{\frac{N+2s}{2}+\tau_1+\frac{(N+2s)(\tau-\tau_1)}{N-2s}}}
\Big(\sum_{j=1}^k\frac{1}{(|\overline{x}_{1}-\underline{x}_{j}|)^{\tau_1}}\Big)^{2^*_s-2}dz\\
&\leqslant C\sum_{j=1}^k\frac{1}{(1+|y-\overline{x}_{j}|)^{\frac{N-2s}{2}+\tau_1+\frac{(N+2s)(\tau-\tau_1)}{N-2s}}}
+C\sum_{j=1}^k\frac{1}{(1+|y-\underline{x}_{j}|)^{\frac{N-2s}{2}+\tau_1+\frac{(N+2s)(\tau-\tau_1)}{N-2s}}},
\end{align*}
where $\frac{N-2s-2}{N-2s}<\tau_1<\tau$.
Since it holds that
\begin{align*}
\frac{N-2s}2+\tau_1+\frac{(N+2s)(\tau-\tau_1)}{N-2s}=\frac{N-2s}2+\tau+\frac{4(\tau-\tau_1)}{N-2s}>\frac{N-2s}2+\tau,
\end{align*}
we  continue this process to obtain the result.
\end{proof}
\smallskip

Recall the linear operator
\begin{align*}
L_k\xi=(-\Delta)^s\xi-(2^*_s-1)K\Big(\frac{|y|}{\mu}\Big)u_k^{2^*_s-2}\xi.
\end{align*}

We prove Theorem \ref{th3} by contradiction. Suppose that there exists some $k\rightarrow+\infty$ such that $\|\xi_n\|_*=1$
and $L_{k}\xi_n=0.$ Let
\begin{align*}
\bar\xi_n(y)=\xi_n(y+\overline{x}_{1}).
\end{align*}

\begin{Lem}\label{lem43}
There exist some constants $b_0$, $b_1$ and $b_3$ such that
\begin{align*}
\bar\xi_n\rightarrow b_0\psi_0+b_1\psi_1+b_3\psi_3,
\end{align*}
uniformly in $C^1(B_R(0))$ for any $R>0$,
$$\psi_0=\frac{\partial U_{0,\Lambda}}{\partial \Lambda}\Big\arrowvert_{\Lambda=1},\ \ \psi_i=\frac{\partial U_{0,1}}{\partial y_i},i=1,3.$$

\end{Lem}

\begin{proof}
Since by assumption,  $|\bar\xi|\leqslant C$, we assume that $\bar\xi\rightarrow\xi$ in $C_{loc}(\R^N)$. It is easy to find that $\xi$ satisfies
\begin{align}\label{44}
(-\Delta)^s\xi=(2^*_s-1)U^{2^*_s-2}\xi,\ \text{ in}\ \R^N,
\end{align}
which gives that
$$\xi=\sum_{i=0}^N b_i\psi_i.$$
Since $\xi_n$ is even in $y_i,i=2,4,\ldots,N$, we obtain that $b_i=0,i=2,4,\ldots,N$.

\end{proof}
\smallskip

Similar to Lemma \ref{lem42}, we show the following result.
\begin{Lem}\label{lem44}
There holds that
\begin{align*}
|\xi_k(y)|\leqslant C\sum_{j=1}^k\frac{1}{(1+|y-\overline{x}_{j}|)^{N-2s}}
+C\sum_{j=1}^k\frac{1}{(1+|y-\underline{x}_{j}|)^{N-2s}},\ \ \text{for}\ \text{all}\ y\in\R^N.
\end{align*}
\end{Lem}

From Lemma \ref{lem43}, we know that $b_{0}$, $b_{1}$ and $b_{3}$ are both bounded.
Now we are ready to prove the following result.
\begin{Lem}\label{lem45}
It holds $b_0=b_1=b_3=0$.
\end{Lem}

\begin{proof}
$\textbf{Step 1.}$  we apply the Pohozaev identities \eqref{42} in the domain $\Omega=B_\delta(\overline{x}_{1})$ and obtain that
\begin{align}\label{45}
&\quad-\int_{\partial''B_\delta^+(\overline{x}_{1})}t^{1-2s}\frac{\partial \tilde{u}_{k}}{\partial\nu}\frac{\partial\tilde{\xi}_{k}}{\partial y_1}
-\int_{\partial''B_\delta^+(\overline{x}_{1})}t^{1-2s}\frac{\partial \tilde{\xi}_{k}}{\partial\nu}\frac{\partial \tilde{u}_{k}}{\partial y_1}
+ \int_{\partial''B_\delta^+(\overline{x}_{1})}t^{1-2s}\langle\nabla \tilde u_{k},\nabla\tilde \xi_{k}\rangle\nu_1\nonumber\\
&=-\int_{B_\delta(\overline{x}_{1})}\frac{\partial K\Big(\frac{|y|}{\mu}\Big)}{\partial y_1}u_{k}^{2^*_s-1}\xi_{k}+\int_{\partial B_\delta(\overline{x}_{1})} K\Big(\frac{|y|}{\mu}\Big)u_{k}^{2^*_s-1}\xi_{k}\nu_1.
\end{align}
It is easy to observe that
\begin{align*}
\int_{\partial B_\delta(\overline{x}_{1})} K\Big(\frac{|y|}{\mu}\Big)u_{k}^{2^*_s-1}\xi_{k}\nu_1=O(\frac{1}{\mu^{m+\sigma}}).
\end{align*}
By Lemma \ref{lemB2}, we  get
\begin{align*}
E_1&:=-\int_{\partial''B_\delta^+(\overline{x}_{1})}t^{1-2s}\frac{\partial \tilde{u}_{k}}{\partial\nu}\frac{\partial\tilde{\xi}_{k}}{\partial y_1}
-\int_{\partial''B_\delta^+(\overline{x}_{1})}t^{1-2s}\frac{\partial \tilde{\xi}_{k}}{\partial\nu}\frac{\partial \tilde{u}_{k}}{\partial y_1}
+ \int_{\partial''B_\delta^+(\overline{x}_{1})}t^{1-2s}\langle\nabla \tilde u_{k},\nabla\tilde \xi_{k}\rangle\nu_1\\
&=-\int_{\partial''B_\delta^+(\overline{x}_{1})}t^{1-2s}\frac{\partial \displaystyle(\sum_{i=1}^k \tilde U_{\overline{x}_{i},\Lambda}+\sum_{i=1}^k \tilde U_{\underline{x}_{i,\Lambda}})}{\partial\nu}\frac{\partial\tilde{\xi}_{k}}{\partial y_1}
-\int_{\partial''B_\delta^+(\overline{x}_{1})}t^{1-2s}\frac{\partial \tilde{\xi}_{k}}{\partial\nu}\frac{\partial \displaystyle(\sum_{i=1}^k \tilde U_{\overline{x}_{i},\Lambda}+\sum_{i=1}^k \tilde U_{\underline{x}_{i,\Lambda}})}{\partial y_1}\\
&\quad+\int_{\partial''B_\delta^+(\overline{x}_{1})}t^{1-2s}\langle\nabla \displaystyle(\sum_{i=1}^k \tilde U_{\overline{x}_{i},\Lambda}+\sum_{i=1}^k \tilde U_{\underline{x}_{i,\Lambda}}),\nabla\tilde \xi_{k}\rangle\nu_1\\
&\leqslant Cb_0\int_{\partial''B_\theta^+(\overline{x}_{1})}t^{1-2s}\frac{1}{|y-\overline{x}_{1}|^{N-2s+1}}
\sum_{j=2}^{k}\frac{(\overline{x}_{j}-\overline{x}_{1})_1}{|\overline{x}_{j}-\overline{x}_{1}|^{N-2s+2}}\nonumber\\
&\quad+Cb_0\int_{\partial''B_\theta^+(\overline{x}_{1})}t^{1-2s}\frac{1}{|y-\overline{x}_{1}|^{N-2s+1}}
\sum_{j=1}^{k}\frac{(\underline{x}_{j}-\overline{x}_{1})_1}{|\underline{x}_{j}-\overline{x}_{1}|^{N-2s+2}}+O(\frac{1}{\mu^{m+\sigma}})\nonumber\\
&\leqslant Cb_0\sum_{j=2}^{k}\frac{(\overline{x}_{j}-\overline{x}_{1})_1}{|\overline{x}_{j}-\overline{x}_{1}|^{N-2s+2}}
+Cb_0\sum_{j=1}^{k}\frac{(\underline{x}_{j}-\overline{x}_{1})_1}{|\underline{x}_{j}-\overline{x}_{1}|^{N-2s+2}}+O(\frac{1}{\mu^{m+\sigma}}).
\end{align*}
Since
$$|\overline{x}_{j}-\overline{x}_{1}|=2r\sqrt{1-h^2}\sin\frac{j\pi}{k}$$
and
$$|\overline{x}_{1}-\underline{x}_{j}|=2r\sqrt{(1-h^2)\sin^2\frac{(j-1)\pi}{k}+h^2}, j=2,3,\cdots,k,$$
we have
$$(\overline{x}_{j}-\overline{x}_{1})_1=(\underline{x}_{j}-\overline{x}_{1})_1=r\sqrt{1-h^2}\cos\frac{2(j-1)\pi}{k}-r\sqrt{1-h^2}=
-\frac{|\overline{x}_{j}-\overline{x}_{1}|^2}{2r\sqrt{1-h^2}}$$
and
$$(\underline{x}_{1}-\overline{x}_{1})_1=0.$$
Thus
\begin{align}\label{46}
E_1
&\leqslant\frac{Cb_0}{r\sqrt{1-h^2}}\Big(\sum_{j=2}^{k}
\frac{1}{|\overline{x}_{j}-\overline{x}_{1}|^{N-2s}}+\sum_{j=2}^{k}
\frac{|\overline{x}_{j}-\overline{x}_{1}|^2}{|\underline{x}_{j}-\overline{x}_{1}|^2}
\frac{1}{|\underline{x}_{j}-\overline{x}_{1}|^{N-2s}}\Big)+O(\frac{1}{\mu^{m+\sigma}})\nonumber\\
&=b_0O\Big(\frac{k^{N-2s}}{(r\sqrt{1-h^2})^{N-2s}}\Big)=b_0O\Big(\frac{1}{\mu^m(\sqrt{1-h^2})^{N-2s}}\Big).
\end{align}

We estimate the right hand side of \eqref{45},
\begin{align}\label{47}
&\quad\int_{B_\delta(\overline{x}_{1})}\frac{\partial K\Big(\frac{|y|}{\mu}\Big)}{\partial y_1}u_{k}^{2^*_s-1}\xi_{k}\nonumber\\
&=\int_{B_\delta(\overline{x}_{1})}\Big(\frac{\partial K\Big(\frac{|y|}{\mu}\Big)}{\partial y_1}-\frac{\partial K\Big(\frac{|\overline{x}_{1}|}{\mu}\Big)}{\partial y_1}\Big)u_{k}^{2^*_s-1}\xi_{k}+O(\frac{1}{\mu^{m+\sigma}})\nonumber\\
&=\int_{B_\delta(\overline{x}_{1})}\Big\langle\nabla\frac{\partial K\Big(\frac{|\overline{x}_{1}|}{\mu}\Big)}{\partial y_1},y-\overline{x}_{1}\Big\rangle
u_{k}^{2^*_s-1}\xi_{k}+O(\frac{1}{\mu^{m+\sigma}})\nonumber\\
&=\int_{\R^N}\Big\langle\nabla\frac{\partial K\Big(\frac{|\overline{x}_{1}|}{\mu}\Big)}{\partial y_1},\frac{y}{\mu\Lambda}\Big\rangle
U^{2^*_s-1}(b_0\psi_0+b_1\psi_1+b_3\psi_3)+O(\frac{1}{\mu^{m+\sigma}})\nonumber\\
&=\frac{K''\Big(\frac{|\overline{x}_{1}|}{\mu}\Big)b_1}{\mu\Lambda}\int_{\R^N}
U^{2^*_s-1}\psi_1y_1+O(\frac{1}{\mu^{m+\sigma}}).
\end{align}

Combining \eqref{46} and \eqref{47}, we  obtain
$$\frac{K''\Big(\frac{|\overline{x}_{1}|}{\mu}\Big)b_1}{\mu\Lambda}\int_{\R^N}
U^{2^*_s-1}\psi_1y_1+O(\frac{1}{\mu^{m+\sigma}})\leqslant b_0O\Big(\frac{1}{\mu^m(\sqrt{1-h^2})^{N-2s}}\Big),$$
which gives $b_1=0$.\\

$\textbf{Step 2.}$ We apply the Pohozaev identities in $y_3$ to get
\begin{align*}
&\quad-\int_{\partial''B_\delta^+(\overline{x}_{1})}t^{1-2s}\frac{\partial \tilde u_{k}}{\partial\nu}\frac{\partial\tilde\xi_{k}}{\partial y_3}
-\int_{\partial''B_\delta^+(\overline{x}_{1})}t^{1-2s}\frac{\partial \tilde\xi_{k}}{\partial\nu}\frac{\partial \tilde  u_{k}}{\partial y_3}
+ \int_{\partial''B_\delta^+(\overline{x}_{1})}t^{1-2s}\langle\nabla \tilde  u_{k},\nabla\tilde \xi_{k}\rangle\nu_3\nonumber\\
&=-\int_{B_\delta(\overline{x}_{1})}\frac{\partial K\Big(\frac{|y|}{\mu}\Big)}{\partial y_3}u_{k}^{2^*_s-1}\xi_{k}+\int_{\partial B_\delta(\overline{x}_{1})} K\Big(\frac{|y|}{\mu}\Big)u_{k}^{2^*_s-1}\xi_{k}\nu_3.
\end{align*}
Again by  Lemma \ref{lem44}
\begin{align*}
\int_{\partial B_\delta(\overline{x}_{1})} K\Big(\frac{|y|}{\mu}\Big)u_{k}^{2^*_s-1}\xi_{k}\nu_3=O(\frac{1}{\mu^{m+\sigma}}).
\end{align*}
Let
\begin{align*} E_3&:=-\int_{\partial''B_\delta^+(\overline{x}_{1})}t^{1-2s}\frac{\partial \tilde u_{k}}{\partial\nu}\frac{\partial\tilde\xi_{k}}{\partial y_3}
-\int_{\partial''B_\delta^+(\overline{x}_{1})}t^{1-2s}\frac{\partial \tilde\xi_{k}}{\partial\nu}\frac{\partial \tilde  u_{k}}{\partial y_3}
+ \int_{\partial''B_\delta^+(\overline{x}_{1})}t^{1-2s}\langle\nabla \tilde  u_{k},\nabla\tilde \xi_{k}\rangle\nu_3\nonumber\\
&=-\int_{\partial''B_\delta^+(\overline{x}_{1})}t^{1-2s}\frac{\partial \displaystyle(\sum_{i=1}^k \tilde U_{\overline{x}_{i},\Lambda}+\sum_{i=1}^k \tilde U_{\underline{x}_{i,\Lambda}})}{\partial\nu}\frac{\partial\tilde{\xi}_{k}}{\partial y_3}
-\int_{\partial''B_\delta^+(\overline{x}_{1})}t^{1-2s}\frac{\partial \tilde{\xi}_{k}}{\partial\nu}\frac{\partial \displaystyle(\sum_{i=1}^k \tilde U_{\overline{x}_{i},\Lambda}+\sum_{i=1}^k \tilde U_{\underline{x}_{i,\Lambda}})}{\partial y_3}\\
&\quad+\int_{\partial''B_\delta^+(\overline{x}_{1})}t^{1-2s}\langle\nabla \displaystyle(\sum_{i=1}^k \tilde U_{\overline{x}_{i},\Lambda}+\sum_{i=1}^k \tilde U_{\underline{x}_{i,\Lambda}}),\nabla\tilde \xi_{k}\rangle\nu_3\\
&\leqslant Cb_0\int_{\partial''B_\theta^+(\overline{x}_{1})}t^{1-2s}\frac{1}{|y-\overline{x}_{1}|^{N-2s+1}}
\sum_{j=1}^{k}\frac{(\underline{x}_{j}-\overline{x}_{1})_3}{|\underline{x}_{j}-\overline{x}_{1}|^{N-2s+2}}\nonumber\\
&\quad+Cb_0\int_{\partial''B_\theta^+(\overline{x}_{1})}t^{1-2s}\frac{1}{|y-\overline{x}_{1}|^{N-2s+1}}
\sum_{j=2}^{k}\frac{(\overline{x}_{j}-\overline{x}_{1})_3}{|\overline{x}_{j}-\overline{x}_{1}|^{N-2s+2}}+O(\frac{1}{\mu^{m+\sigma}})\nonumber\\
&\leqslant Cb_0\sum_{j=1}^{k}\frac{(\overline{x}_{1}-\underline{x}_{j})_3}{|\overline{x}_{1}-\underline{x}_{j}|^{N-2s+2}}
+Cb_0\sum_{j=2}^{k}\frac{(\overline{x}_{j}-\overline{x}_{1})_3}{|\overline{x}_{j}-\overline{x}_{1}|^{N-2s+2}}+O(\frac{1}{\mu^{m+\sigma}}).
\end{align*}
Since $(\underline{x}_{1}-\overline{x}_{1})_3=|\underline{x}_{1}-\overline{x}_{1}|=-2rh$ and $(\overline{x}_{j}-\overline{x}_{1})_3=0$,
then
\begin{align}\label{48}
E_3
&\leqslant\sum_{j=1}^{k}\frac{2Cb_0rh}{|\underline{x}_{j}-\overline{x}_{1}|^{N-2s+2}}
+O(\frac{1}{\mu^{m+\sigma}})\nonumber\\
&=b_0O\Big(\frac{k}{r^{N-2s+1}h^{N-2s}\sqrt{1-h^2}}\Big)=b_0O\Big(\frac{1}{\sqrt{1-h^2}\mu^{\frac{4(N-2s)^2+2(N-2s+1)-2(N-2s)}{(N-2s)(N-2s+1)}}}\Big).
\end{align}
Similar to \eqref{47}, we have
\begin{align}\label{49}
\int_{B_\delta(\overline{x}_{1})}\frac{\partial K\Big(\frac{|y|}{\mu}\Big)}{\partial y_3}u_{k}^{2^*_s-1}\xi_{k}=\frac{K''\Big(\frac{|\overline{x}_{1}|}{\mu}\Big)b_3}{\mu\Lambda}\int_{\R^N}
U^{2^*_s-1}\psi_3y_3+O(\frac{1}{\mu^{m+\sigma}}).
\end{align}

Combining \eqref{48} and \eqref{49}, we  obtain
$$\frac{K''\Big(\frac{|\overline{x}_{1}|}{\mu}\Big)b_3}{\mu\Lambda}\int_{\R^N}
U^{2^*_s-1}\psi_3y_3+O(\frac{1}{\mu^{m+\sigma}})\leqslant O\Big(\frac{1}{\mu^{\frac{(N-2s+1)^2-(N-2s+1)(N-2s-m)}{(N-2s+1)}}\sqrt{1-h^2}}\Big).$$
Since $$\frac{(N-2s+1)^2-(N-2s+1)(N-2s-m)}{(N-2s+1)}>1,$$
thus $b_3=0$.

$\textbf{Step 3.}$ We use \eqref{43} for $\Omega=\R^N$, from Lemma \ref{lemA11}, we obtain
$$\int_{\R^N}u_{k}^{2^*_s-1}\xi_{k}\langle\nabla K\Big(\frac{|y|}{\mu}\Big),y\rangle=0,$$
which gives
\begin{align*}
\int_{\Omega_1^+}u_{k}^{2^*_s-1}\xi_{k}\langle\nabla K\Big(\frac{|y|}{\mu}\Big),y\rangle=0.
\end{align*}
On the other hand, we have
\begin{align*}
&\quad\int_{\Omega_1^+}u_{k}^{2^*_s-1}\xi_{k}\langle\nabla K\Big(\frac{|y|}{\mu}\Big),y\rangle=\int_{B_\delta(\overline{x}_{1})}u_{k}^{2^*_s-1}\xi_{k}\langle\nabla K\Big(\frac{|y|}{\mu}\Big),y\rangle+O(\frac{1}{\mu^{m+\sigma}})\nonumber\\
&=\int_{B_\delta(\overline{x}_{1})}u_{k}^{2^*_s-1}\xi_{k}\langle\nabla K\Big(\frac{|y|}{\mu}\Big),y-\overline{x}_{1}\rangle+r\int_{B_\delta(\overline{x}_{1})}u_{k}^{2^*_s-1}\xi_{k}\frac{\partial K\Big(\frac{|y|}{\mu}\Big)}{\partial y_1}+O(\frac{1}{\mu^{m+\sigma}}),
\end{align*}
which gives
\begin{align*}
-r\int_{B_\delta(\overline{x}_{1})}u_{k}^{2^*_s-1}\xi_{k}\frac{\partial K\Big(\frac{|y|}{\mu}\Big)}{\partial y_1}=\int_{B_\delta(\overline{x}_{1})}u_{k}^{2^*_s-1}\xi_{k}\langle\nabla K\Big(\frac{|y|}{\mu}\Big),y-\overline{x}_{1}\rangle+O(\frac{1}{\mu^{m+\sigma}}).
\end{align*}
However, we have
\begin{align}\label{411}
&\quad\int_{B_\delta(\overline{x}_{1})}u_{k}^{2^*_s-1}\xi_{k}\langle\nabla K\Big(\frac{|y|}{\mu}\Big),y-\overline{x}_{1}\rangle\nonumber\\
&=\int_{B_\delta(\overline{x}_{1})}u_{k}^{2^*_s-1}\xi_{k}\langle\nabla K\Big(\frac{|y|}{\mu}\Big)-\nabla K\Big(\frac{|\overline{x}_{1}|}{\mu}\Big),y-\overline{x}_{1}\rangle+O(\frac{1}{\mu^{m+\sigma}})\nonumber\\
&=\int_{B_\delta(\overline{x}_{1})}u_{k}^{2^*_s-1}\xi_{k}\langle\nabla^2 K\Big(\frac{|\overline{x}_{1}|}{\mu}\Big)\frac{y-\overline{x}_{1}}{\mu},y-\overline{x}_{1}\rangle+O(\frac{1}{\mu^{m+\sigma}})\nonumber\\
&=\frac{b_0\Delta K\Big(\frac{|\overline{x}_{1}|}{\mu}\Big)}{N\mu}\int_{\R^N}U^{2^*_s-1}\psi_0|y|^2+o(\frac{1}{\mu^{m}}).
\end{align}
Thus we obtain
\begin{align}\label{412}
\int_{B_\delta(\overline{x}_{1})}u_{k}^{2^*_s-1}\xi_{k}\frac{\partial K\Big(\frac{|y|}{\mu}\Big)}{\partial y_1}=-\frac{b_0\Delta K\Big(\frac{|\overline{x}_{1}|}{\mu}\Big)}{N\mu^2}\int_{\R^N}U^{2^*_s-1}\psi_0|y|^2+o(\frac{1}{\mu^{m}}).
\end{align}

Combining \eqref{47} and \eqref{412}, we have
\begin{align*}
&\quad-\frac{b_0\Delta K\Big(\frac{|\overline{x}_{1}|}{\mu}\Big)}{N\mu^2}\int_{\R^N}U^{2^*_s-1}\psi_0|y|^2+o(\frac{1}{\mu^{m}})\leqslant b_0O\Big(\frac{1}{\mu^m(\sqrt{1-h^2})^{N-2s}}\Big).
\end{align*}
Thus $b_0=0$.
\end{proof}

\medskip

\section*{Appendix}

\appendix

\section{Basic Estimates}
\renewcommand{\theequation}{A.\arabic{equation}}

In this section, we give some basic estimates.

For each fixed $i$ and $j,i\ne j$, consider the following function
$$g_{ij}(y)=\frac{1}{(1+|x-\overline{x}_{i}|)^{\alpha}}\frac{1}{(1+|x-\overline{x}_{j}|)^{\beta}},$$
where $\alpha\geqslant 1$ and $\beta\geqslant 1$ are two constants.
\begin{Lem}[c.f. \textbf{\cite{JS}}]\label{lemA1}
For any constant $0<\sigma\leqslant\min\{\alpha,\beta\}$, there is a constant $C>0$, such that
$$g_{ij}(y)\leqslant\frac{C}{|\overline{x}_{i}-\overline{x}_{j}|^{\sigma}}
\Big(\frac{1}{(1+|x-\overline{x}_{i}|)^{\alpha+\beta-\sigma}}+\frac{1}{(1+|x-\overline{x}_{j}|)^{\alpha+\beta-\sigma}}\Big).$$
\end{Lem}
\begin{Lem}[c.f. \textbf{\cite{YJ}}]\label{lemA2}
For any constant $0<\sigma<N-2s$, there is a constant $C>0$, such that
$$\int_{\R^N}\frac{1}{|y-z|^{N-2s}}\frac{1}{(1+|z|)^{2s+\sigma}}dz\leqslant\frac{C}{(1+|y|)^\sigma}.$$
\end{Lem}

\begin{Lem}\label{lemA3}
Suppose $\tau\in(0,\frac{N-2s}{2}), y=(y_1,y_2,\cdots,y_N)$. Then there is a small $\theta>0$, such that when $y_3\geqslant0$,
\begin{align*}
\int_{\R^N}&\frac{1}{|y-z|^{N-2s}}W^{\frac{4s}{N-2s}}_{r,h,\Lambda}\sum_{j=1}^k
\frac{1}{(1+|z-\overline{x}_{j}|)^{\frac{N-2s}{2}+\tau}}dz
\leqslant C\sum_{j=1}^k\frac{1}{(1+|y-\overline{x}_{j}|)^{\frac{N-2s}{2}+\tau+\theta}},
\end{align*}
and when $y_3<0$,
\begin{align*}
\int_{\R^N}&\frac{1}{|y-z|^{N-2s}}W^{\frac{4s}{N-2s}}_{r,h,\Lambda}\sum_{j=1}^k
\frac{1}{(1+|z-\underline{x}_{j}|)^{\frac{N-2s}{2}+\tau}}dz
\leqslant C\sum_{j=1}^k\frac{1}{(1+|y-\underline{x}_{j}|)^{\frac{N-2s}{2}+\tau+\theta}}.
\end{align*}
\end{Lem}
\begin{proof}
The proof of Lemma \ref{lemA3} is similar to \cite{YJ}, we omit it here.
\end{proof}
\begin{Lem}\label{lemA4}\it For any $\tau>0$, there is a constant $C>0$, such that
\begin{align}\label{a1}
\sum_{i=2}^k\frac{1}{|\overline{x}_{1}-\overline{x}_{i}|^{\tau}}\leqslant\frac{Ck^\tau}{(r\sqrt{1-h^2})^{\tau}}\sum_{i=2}^k\frac{1}{i^\tau}
&=\begin{cases}
\dfrac{Ck^\tau}{(r\sqrt{1-h^2})^{\tau}}\Big(1+O(\frac{1}{k^2})\Big),&\text{if} \ \tau>1,\\
\dfrac{Ck^\tau}{(r\sqrt{1-h^2})^{\tau}}\Big(1+O(\frac{\ln k}{k^2})\Big),&\text{if}\ \tau\leqslant1,
\end{cases}
\end{align}
and
\begin{align}\label{a2}
&\quad\quad\sum_{i=1}^k\frac{1}{|\overline{x}_{1}-\underline{x}_{i}|^{\tau}}
=\frac{1}{(rh)^\tau}\frac{D_1hk}{\sqrt{1-h^2}}\Big(1+o(\frac{1}{hk})\Big)+\begin{cases}
O\Big(\dfrac{Ck^\tau}{(r\sqrt{1-h^2})^{\tau}}\Big),&\text{if} \ \tau>1,\\
O\Big(\dfrac{Ck\ln k}{(r\sqrt{1-h^2})^{\tau}}\Big),&\text{if}\ \tau\leqslant1,
\end{cases}
\end{align}
where $D_1=\int_0^{+\infty}\frac{1}{(x^2+1)^{\frac{\tau}{2}}}dx$.
\end{Lem}
\begin{proof} Recall
\begin{align*}
&\overline{x}_{1}=r\Bigl(\sqrt{1-h^2},0,h,\textbf{0}\Bigl),\quad
\overline{x}_{i}=r\Bigl(\sqrt{1-h^2}\cos\frac{2(i-1)\pi}{k},\sqrt{1-h^2}\sin\frac{2(i-1)\pi}{k},h,\textbf{0}\Bigl),\nonumber\\
&\overline{x}_{i}-\overline{x}_{1}=r\Bigl(\sqrt{1-h^2}\cos\frac{2(i-1)\pi}{k}-\sqrt{1-h^2},\sqrt{1-h^2}\sin\frac{2(i-1)\pi}{k},0,\textbf{0}\Bigl).\nonumber
\end{align*}
We get
$|\overline{x}_{i}-\overline{x}_{1}|=2r\sqrt{1-h^2}\sin\frac{(i-1)\pi}{k}.$
For any $\tau>0$, it holds
\begin{align}
\sum_{i=2}^k\frac{1}{|\overline{x}_{1}-\overline{x}_{i}|^{\tau}}&=\frac{1}{(2r\sqrt{1-h^2})^{\tau}}\sum_{i=2}^k\frac{1}{(\sin\frac{(i-1)\pi}{k})^{\tau}}\nonumber\\
&=\begin{cases}
\dfrac{1}{(2r\sqrt{1-h^2})^{\tau}}\sum_{i=2}^\frac{k}{2}\dfrac{1}{(\sin\frac{(i-1)\pi}{k})^{\tau}}+\dfrac{1}{(2r\sqrt{1-h^2})^{\tau}},&\text{if} \ \text{k} \ \text{is}\ \text{even},\\
\dfrac{1}{(2r\sqrt{1-h^2})^{\tau}}\sum_{i=2}^{[\frac{k}{2}]}\dfrac{1}{(\sin\frac{(i-1)\pi}{k})^{\tau}},&\text{if}\ \text{k}\ \text{is} \ \text{odd}.\nonumber
\end{cases}
\end{align}
Note that there exist two constants $C_1,C_2>0$, such that
$$0<C_{1}\leqslant\frac{\sin\frac{(i-1)\pi}{k}}{\frac{(i-1)\pi}{k}}\leqslant C_{2},\ \ i=2,\cdots,[\frac{k}{2}]+1.$$
Thus for any $\tau>\frac{N-2s-m}{N-2s}$, we have \eqref{a1}.

\smallskip

On the other hand, we get
$$|\overline{x}_{1}-\underline{x}_{i}|=2r\sqrt{(1-h^2)\sin^2\frac{(i-1)\pi}{k}+h^2}.$$
Similarly, we have
\begin{align*}
\sum_{i=1}^k\frac{1}{|\overline{x}_{1}-\underline{x}_{i}|^{\tau}}&=\sum_{i=1}^k\frac{1}{\Big(2r\sqrt{(1-h^2)\sin^2\frac{(i-1)\pi}{k}+h^2}\Big)^\tau}\\
&=\frac{2}{(2rh)^\tau}\sum_{i=1}^{\frac{[k]}{2}+1}\frac{1}{\Big(1+\frac{1-h^2}{h^2}
\frac{(i-1)^2\pi^2}{k^2}\Big)^{\frac{\tau}{2}}}+\begin{cases}
O\Big(\dfrac{Ck^\tau}{(r\sqrt{1-h^2})^{\tau}}\Big),&\text{if} \ \tau>1,\\
O\Big(\dfrac{Ck\ln k}{(r\sqrt{1-h^2})^{\tau}}\Big),&\text{if}\ \tau\leqslant1.
\end{cases}\\
\end{align*}
We obtain
\begin{align*}
\sum_{i=1}^{\frac{[k]}{2}+1}\frac{1}{\Big(1+\frac{1-h^2}{h^2}
\frac{(i-1)^2\pi^2}{k^2}\Big)^{\frac{\tau}{2}}}&=\int_0^{\frac{[k]}{2}+1}\frac{1}{\Big(1+\frac{1-h^2}{h^2}
\frac{x^2\pi^2}{k^2}\Big)^{\frac{\tau}{2}}}dx+o(1)\\
&=\frac{hk}{\sqrt{1-h^2}\pi}\int_0^{\frac{\pi^2(1-h^2)}{4h^2}}\frac{1}{(x^2+1)^{\frac{\tau}{2}}}dx+o(1)\\
&=\frac{hk}{\sqrt{1-h^2}\pi}\int_0^{+\infty}\frac{1}{(x^2+1)^{\frac{\tau}{2}}}dx\Big(1+o(\frac{1}{hk})\Big).\\
\end{align*}
Thus, it holds \eqref{a2}.
\end{proof}

\begin{Lem}[c.f. \textbf{\cite{GNNT}}]\label{lemB1}
Let $\theta>0$ is a constant. Suppose that $(y-x)^2+t^2\geqslant\delta^2,t>0$ and $\alpha>N$. Then, when
$0<\beta<N$, we have
\begin{align*}
\int_{\R^N}\frac{1}{(t+|z|)^\alpha}\frac{1}{|y-z-x|^\beta}dz\leqslant
C\Big(\frac{1}{(1+|y-x|)^\beta}\frac{1}{t^{\alpha-N}}+\frac{1}{(1+|y-x|)^{\alpha+\beta-N}}\Big).
\end{align*}
\end{Lem}
\begin{Lem}\label{lemB2}
Suppose that $(y-x)^2+t^2=\delta^2,t>0$,  then there is a constant $C>0$, such that
\begin{align}\label{B1}
\Big\vert\nabla \tilde U_{x,\Lambda}\Big\vert\leqslant\frac{C}{(1+|y-x|)^{N-2s+1}},
\end{align}
and
\begin{align}\label{B2}\Big\vert\nabla \tilde \psi_{i}\Big\vert\leqslant\frac{C}{(1+|y-x|)^{N-2s+1}},\quad i=0,1,2,\cdots,N.\end{align}
\end{Lem}
\begin{proof}
Since $U_{x,\Lambda}=C(N,s)\frac{\Lambda^{\frac{N-2s}{2}}}{(1+\Lambda^2|y-x|^2)^{\frac{N-2s}{2}}}$,
we have
\begin{align*}
\tilde U_{x,\Lambda}(y,t)&=\beta(N,s)\int_{\R^N}\frac{t^{2s}}{(|y-\xi|^2+t^2)
^{\frac{N+2s}{2}}}U_{x,\Lambda}d\xi\\
&=\beta(N,s)C(N,s)\int_{\R^N}\frac{t^{2s}}{(|y-\xi|^2+t^2)
^{\frac{N+2s}{2}}}\frac{\Lambda^{\frac{N-2s}{2}}}{(1+\Lambda^2|\xi-x|^2)^{\frac{N-2s}{2}}}d\xi.
\end{align*}
Note that for $i=1,\cdots,N$,
\begin{align*}
&\quad\frac{\partial}{\partial y_i}\int_{\R^N}\frac{t^{2s}}{(|y-\xi|^2+t^2)
^{\frac{N+2s}{2}}}\frac{\Lambda^{\frac{N-2s}{2}}}{(1+\Lambda^2|\xi-x|^2)^{\frac{N-2s}{2}}}d\xi\\
&=\frac{\partial}{\partial y_i}\int_{\R^N}\frac{1}{(1+|z|^2)
^{\frac{N+2s}{2}}}\frac{\Lambda^{\frac{N-2s}{2}}}{(1+\Lambda^2|y-tz-x|^2)^{\frac{N-2s}{2}}}dz\\
&\leqslant C\int_{\R^N}\frac{1}{(1+|z|^2)
^{\frac{N+2s}{2}}}\frac{\partial}{\partial y_i}\frac{1}{(1+|y-tz-x|^2)^{\frac{N-2s}{2}}}dz\\
&\leqslant C\int_{\R^N}\frac{1}{(1+|z|^2)
^{\frac{N+2s}{2}}}\frac{(y-tz-x)_i}{(1+|y-tz-x|^2)^{\frac{N-2s}{2}+1}}dz
\end{align*}
and
\begin{align*}
&\quad\frac{\partial}{\partial t}\int_{\R^N}\frac{t^{2s}}{(|y-\xi|^2+t^2)
^{\frac{N+2s}{2}}}\frac{\Lambda^{\frac{N-2s}{2}}}{(1+\Lambda^2|\xi-x|^2)^{\frac{N-2s}{2}}}d\xi\\
&\leqslant C\int_{\R^N}\frac{1}{(1+|z|^2)
^{\frac{N+2s}{2}}}\frac{\sum_{i=1}^N(y-tz-x)_iz_i}{(1+|y-tz-x|^2)^{\frac{N-2s}{2}+1}}dz.
\end{align*}
By Lemma \ref{lemB1}, we  get
\begin{align*}
|\nabla \tilde U_{x,\Lambda}|&\leqslant
C\int_{\R^N}\frac{1}{(1+|z|)^{N+2s-1}}\frac{1}{(1+|y-tz-x|)^{N-2s+1}}dz\\
&\leqslant C\int_{\R^N}\frac{t^{2s-1}}{(t+|z|)^{N+2s-1}}\frac{1}{(1+|y-z-x|)^{N-2s+1}}dz\\
&\leqslant Ct^{2s-1}\Big(\frac{1}{(1+|y-x|)^{N-2s+1}}\frac{1}{t^{2s-1}}+\frac{1}{(1+|y-x|)^{N}}\Big)\\
&\leqslant \frac{C}{(1+|y-x|)^{N-2s+1}}.
\end{align*}
The proof of \eqref{B2} is similar to  \eqref{B1}, so we omit it  here.
\end{proof}

In order to apply the Pohozaev identities in the unbounded domain in $\mathbb R_+^{N+1}$, it is necessary to verify  the following result.
\begin{Lem}\label{lemA11}
Suppose $\Omega$ is a unbounded domain, we have that the following integrals are finite, i.e.,
\begin{align*}
\Big\vert\int_{\overline \Omega\times(M,+\infty)}t^{1-2s}\frac{\partial\tilde u}{\partial\nu}\frac{\partial\tilde \xi}{\partial y_i}+t^{1-2s}\frac{\partial\tilde \xi}{\partial\nu}\frac{\partial\tilde u}{\partial y_i}\Big\vert,\quad \int_{\overline \Omega\times(M,+\infty)}t^{1-2s}|\nabla\tilde u||\nabla\tilde \xi|,\quad \Big\vert\int_{\overline \Omega\times(M,+\infty)}t^{1-2s}\left<\nabla\tilde u,\nabla\tilde \xi\right>\left<Y-(X_0,0),\nu\right>\Big\vert,\\
\Big\vert\int_{\overline \Omega\times(M,+\infty)}t^{1-2s}\frac{\partial\tilde u}{\partial\nu}\tilde \xi+t^{1-2s}\frac{\partial\tilde \xi}{\partial\nu}\tilde u\Big\vert,\quad
\Big\vert\int_{\overline \Omega\times(M,+\infty)}t^{1-2s}\frac{\partial\tilde u}{\partial\nu}\left<\nabla\tilde\xi,Y-(X_0,0)\right>+t^{1-2s}\frac{\partial\tilde \xi}{\partial\nu}\left<\nabla\tilde u,Y-(X_0,0)\right>\Big\vert<+\infty.
\end{align*}
\end{Lem}
\begin{proof}
It suffices to  prove that
\begin{align*}
&\int_{\overline \Omega\times(M,+\infty)}t^{1-2s}\vert\nabla\tilde u\vert|\nabla\tilde \xi|<+\infty,\quad \int_{\overline \Omega\times(M,+\infty)}t^{1-2s}|\nabla\tilde u||\nabla\tilde \xi|\vert|y-X_0|+t\vert<+\infty,\\&
\int_{\overline \Omega\times(M,+\infty)}t^{1-2s}|\nabla\tilde u||\tilde \xi|<+\infty,\quad \int_{\overline \Omega\times(M,+\infty)}t^{1-2s}|\nabla\tilde \xi||\tilde u|<+\infty.
\end{align*}
We only give the proof for $s\leqslant\frac{1}{2}$, the case $s>\frac{1}{2}$ can be similarly estimated. From definition, we have
\begin{align*}
&\quad\int_{\overline \Omega\times(M,+\infty)}t^{1-2s}\vert\nabla\tilde u\vert|\nabla\tilde \xi|||y-X_0|+t|dtdy\\
&\leqslant\int_{\overline \Omega}\int_M^{+\infty}t^{1-2s}||y-X_0|+t|\Big(\int_{\R^N}\frac{t^{2s-1}|u(x)|}{(|y-x|^2+t^2)^\frac{N+2s}{2}}dx\Big)\Big(
\int_{\R^N}\frac{t^{2s-1}|\xi(x)|}{(|y-x|^2+t^2)^\frac{N+2s}{2}}dx\Big)dtdy\\
&=\int_{\overline \Omega}\int_M^{+\infty}\Big(\int_{\R^N}\frac{t^{\frac{2s-1}{2}}||y-X_0|+t|^\frac{1}{2}|u(x)|}{(|y-x|+t)^{N+2s}}dx\Big)\Big(
\int_{\R^N}\frac{t^{\frac{2s-1}{2}}||y-X_0|+t|^\frac{1}{2}|\xi(x)|}{(|y-x|+t)^{N+2s}}dx\Big)dtdy\\
&\leqslant\int_{\overline \Omega}\Big(\int_{\R^N}\Big(\int_M^{+\infty}\frac{t^{2s-1}||y-X_0|+t|}{(|y-x|+t)^{2N+4s}}dt\Big)^\frac{1}{2}|u(x)|dx\Big)
\Big(\int_{\R^N}\Big(\int_M^{+\infty}\frac{t^{2s-1}||y-X_0|+t|}{(|y-x|+t)^{2N+4s}}dt\Big)^\frac{1}{2}|\xi(x)|dx\Big)dy\\
&\leqslant\int_{\overline \Omega}\Big(\int_{\R^N}\Big(\frac{|y-X_0|^\frac{1}{2}}{(|y-x|+M)^{N+s}}+
\frac{1}{(|y-x|+M)^{N+s-\frac{1}{2}}}\Big)\frac{1}{(1+|x|)^{N-2s}}dx\Big)^2dy\\
&\leqslant\int_{\overline \Omega}\Big(\Big(\int_{\R^N}\Big(\frac{|y-X_0|^\frac{1}{2}}{(|y-x|+M)^{N+s}}\frac{1}{(1+|x|)^{N-2s}}dx\Big)^2
+\Big(\int_{\R^N}\frac{1}{(|y-x|+M)^{N+s-\frac{1}{2}}}\frac{1}{(1+|x|)^{N-2s}}dx\Big)^2\Big)dy\\
&\displaystyle\leqslant\begin{cases}
\displaystyle\int_{\overline \Omega}\frac{|y-X_0|}{(|y|+1)^{2N-4s}}+\frac{(\ln|y|)^2}{(|y|+1)^{2N-4s}}dy,&\text{if} \ s=\frac{1}{2},\\
\displaystyle\int_{\overline \Omega}\frac{|y-X_0|}{(|y|+1)^{2N-4s}}+\frac{1}{(|y|+1)^{2N-2s-1}}dy,&\text{if}\ s<\frac{1}{2},
\end{cases}\\
&\displaystyle\leqslant\int_{\overline \Omega}\frac{1}{(|y|+1)^{2N-4s-1}}dy<+\infty.
\end{align*}
In fact, the last but one inequality for $s=\frac{1}{2}$ holds as follows.
 It suffices to obtain the estimate for $|y|\geqslant2$. Let $d=\frac{1}{2}|y|$, there holds that
$$\int_{B_d(0)}\frac{1}{(1+|y-z|)^{N-2s}}\frac{1}{(1+|z|)^N
}dz\leqslant\frac{1}{(1+|y|)^{N-2s}}\int_{B_d(0)}\frac{1}{(1+|z|)^Ndz
}\leqslant\frac{\ln |y|}{(1+|y|)^{N-2s}},$$
and
$$\int_{B_d(y)}\frac{1}{(1+|y-z|)^{N-2s}}\frac{1}{(1+|z|)^N
}dz\leqslant\frac{1}{(1+|y|)^{N}}\int_{B_d(y)}\frac{1}{(1+|y-z|)^{N-2s}dz
}\leqslant\frac{\ln |y|}{(1+|y|)^{N-2s}}.$$
When $z\in\R^n\backslash(B_d(0)\cup B_d(y))$, we have $|z-y|\geqslant\frac{1}{2}|y|$ and
$|z|\geqslant\frac{1}{2}|y|$.\\
If $|z|\geq2|y|$, we get $|z-y|\geq|z|-|y|\geq\frac12|z|$
and
\begin{align*}
\frac1{(1+|y-z|)^{N-2s}}\frac1{(1+|z|)^{N}}\leqslant \frac1{1+|z|^{N-2s}}\frac1{(1+|z|)^{N}}.
\end{align*}
If $|z|\leqslant2|y|$, then
\begin{align*}
\frac1{(1+|y-z|)^{N-2s}}\frac1{(1+|z|)^{N}}\leqslant \frac1{1+|y|^{N-2s}}\frac1{(1+|z|)^{N}}\leqslant \frac1{1+|z|^{N-2s}}\frac1{(1+|z|)^{N}}.
\end{align*}
As a result,
\begin{align*}
\frac1{(1+|y-z|)^{N-2s}}\frac1{(1+|z|)^{N}}\leq \frac1{(1+|z|)^{N-2s}}\frac1{(1+|z|)^{N}},\ z\in\R^N\setminus(B_d(0)\cup B_d(y)),
\end{align*}
which implies that
\begin{align*}
&\quad\int_{\R^N\setminus(B_d(0)\cup B_d(y))}\frac1{(1+|y-z|)^{N-2s}}\frac1{(1+|z|)^{N}}dz\\
&\leq \int_{\R^N\setminus(B_d(0)\cup B_d(y))}\frac1{(1+|z|)^{N-2s}}\frac1{(1+|z|)^{N}}\leq \frac 1{(1+|z|)^{N-2s}}.
\end{align*}
Other terms can be handled similarly and we omit it  here.
\end{proof}

\medskip
\section{Energy expansion }

\renewcommand{\theequation}{B.\arabic{equation}}
\begin{Lem}\label{lemA5} There is a small $\epsilon>0$, such that

\begin{align*}
I(W_{r,h,\Lambda})&=k\Big(A+\frac{A_1}{\Lambda^m \mu^m}+\frac{A_2}{\Lambda^{m-2}\mu^m}(\mu r_{0}-r)^2
-\frac{B_{1}k^{N-2s}}{\Lambda^{N-2s}(r\sqrt{1-h^2})^{N-2s}}
-\frac{B_2}{\Lambda^{N-2s}(rh)^{N-2s}}\frac{hk}{\sqrt{1-h^2}}\nonumber\\
&\quad+\frac{C}{\mu^m}|\mu r_0-r|^{2+\theta}+\frac{1}{\mu^{m+\theta}}
+O\Big(\frac{1}{\mu^{N-2s-\frac{N-2s-m}{N-2s}-\frac{(N-2s-m)(N-2s-1)^2}{(N-2s)(N-2s+1)}+\theta}}\Big)\Big),
\end{align*}
where $ A, A_1, A_2, B_{1}\ and\ B_{2}$ are some positive constants.
\end{Lem}
\begin{proof}
Recall
$$ I(u)=\frac{1}{2}\int_{\R^N}|(-\Delta)^{\frac{s}{2}} u|^2
-\frac{1}{2^{*}_s}\int_{\R^N}K\Big(\frac{|y|}{\mu}\Big)|u|^{2^{*}_s}.$$
We should calculate
\begin{align*} I(W_{r,h,\Lambda})&=\frac{1}{2}\int_{\R^N}|(-\Delta)^{\frac{s}{2}} W_{r,h,\Lambda}|^2
-\frac{1}{2^{*}_s}\int_{\R^N}K\Big(\frac{|y|}{\mu}\Big)|W_{r,h,\Lambda}|^{2^{*}_s}\\
&:=I_1+I_2.
\end{align*}

First, by symmetry, we obtain
\begin{align}\label{A4}
\frac{1}{2}\int_{\R^N}|(-\Delta)^{\frac{s}{2}} W_{r,h,\Lambda}|^2
&=\frac{1}{2}\sum_{i=1}^k\sum_{j=1}^k\int_{\R^N}\Big(U^{2^{*}_s-1}_{\overline{x}_{j,\Lambda}}+U^{2^{*}_s-1}_{\underline{x}_{j,\Lambda}}\Big)\Big(U_{\overline{x}_{i},\Lambda}+U_{\underline{x}_{i,\Lambda}}\Big)\nonumber\\
&=\frac{1}{2}\sum_{i=1}^k\sum_{j=1}^k\int_{\R^N}\Big(U^{2^{*}_s-1}_{\overline{x}_{j,\Lambda}}U_{\overline{x}_{i},\Lambda}+U^{2^{*}_s-1}_{\overline{x}_{j,\Lambda}}U_{\underline{x}_{i,\Lambda}}+U^{2^{*}_s-1}_{\underline{x}_{j,\Lambda}}U_{\overline{x}_{i},\Lambda}+U^{2^{*}_s-1}_{\underline{x}_{j,\Lambda}}U_{\underline{x}_{i,\Lambda}}\Big)\nonumber\\
&=k\int_{\R^N}U^{2^{*}_s}+k\sum_{i=2}^k\int_{\R^N}U^{2^{*}_s-1}_{\overline{x}_{1,\Lambda}}U_{\overline{x}_{i},\Lambda}+k\sum_{j=1}^k\int_{\R^N}U^{2^{*}_s-1}_{\underline{x}_{1,\Lambda}}U_{\overline{x}_{j},\Lambda}\nonumber\\
&=k\int_{\R^N}U^{2^{*}_s}+k\sum_{j=2}^k\frac{B_{0}}{\Lambda^{N-2s}|\overline{x}_{1}-\overline{x}_{j}|^{N-2s}}+
k\sum_{j=1}^k\frac{B_{1}}{\Lambda^{N-2s}|\overline{x}_{1}-\underline{x}_{j}|^{N-2s}}\nonumber\\
&\quad+O\Big(k\sum_{j=2}^k\frac{1}{|\overline{x}_{1}-\overline{x}_{j}|^{N-2s+\sigma}}+k\sum_{j=1}^k\frac{1}{|\overline{x}_{1}-\underline{x}_{j}|^{N-2s+\sigma}}\Big).
\end{align}
Next, we consider $I_2$. By Lemma \ref{lemA1},
\begin{align*}
\int_{\R^N}K\Big(\frac{|y|}{\mu}\Big)|W_{r,h,\Lambda}|^{2^{*}_s}&=2k\int_{\Omega^{+}_1}K\Big(\frac{|y|}{\mu}\Big)\Big(U_{\overline{x}_{1},\Lambda}+U_{\underline{x}_{1},\Lambda}+\sum_{j=2}^k U_{\overline{x}_{j},\Lambda}+\sum_{j=2}^k U_{\underline{x}_{j},\Lambda}\Big)^{2^{*}_s}\nonumber\\
&=2k\Big(\int_{\Omega^{+}_1}K\Big(\frac{|y|}{\mu}\Big)U_{\overline{x}_{1},\Lambda}^{2^{*}_s}+2^*_sK\Big(\frac{|y|}{\mu}\Big)U_{\overline{x}_{1},\Lambda}^{2^{*}_s-1}(\sum_{j=2}^k U_{\overline{x}_{j},\Lambda}+\sum_{j=1}^k U_{\underline{x}_{j},\Lambda})\nonumber\\
&\quad+O\Big(U_{\overline{x}_{1},\Lambda}^{2^*_s/2}(\sum_{j=2}^k U_{\overline{x}_{j},\Lambda}+\sum_{j=1}^k U_{\underline{x}_{j},\Lambda})^{2^*_s/2}\Big)\nonumber\\
&:=2k(I_{21}+I_{22}+I_{23}).
\end{align*}
For $I_{21}$, we divide the space $\Omega^+_1$ into two parts, namely,
$$K_1:=\{y\in\Omega^+_1:\arrowvert|y|-\mu r_0\arrowvert\geqslant\delta\mu\}$$
and
$$K_2:=\{y\in\Omega^+_1:\arrowvert|y|-\mu r_0\arrowvert<\delta\mu\},$$
where $\delta$ is the constant in $K(|y|)$.
For the space $K_1$, we have
$$|y-\overline{x}_{1}|\geqslant||y|-|\overline{x}_{1}||\geqslant||y|-\mu r_0|-||\overline{x}_{1}|-\mu r_0||\geqslant\frac{1}{2}\delta\mu.$$
So
\begin{align*}
\int_{K_1}[K\Big(\frac{|y|}{\mu}\Big)-1]U_{\overline{x}_{1},\Lambda}^{2^{*}_s}
&\leqslant C\int_{K_1}\Big(\frac{1}{1+|y-\overline{x}_{1}|}\Big)^{2N}\\
&\leqslant \frac{C}{\mu^{N-\tau}}\int_{K_1}\Big(\frac{1}{1+|y-\overline{x}_{1}|}\Big)^{N+\tau}\\
&=O(\frac{C}{\mu^{N-\tau}}).
\end{align*}
For the space $K_2$, by the property of $K(y)$, we  obtain that
\begin{align*}
\int_{K_2}[K\Big(\frac{|y|}{\mu}\Big)-1]U_{\overline{x}_{1},\Lambda}^{2^{*}_s}&=
-\frac{c_{0}}{\mu^m}\int_{K_2}||y|-\mu r_0|^m U_{\overline{x}_{1},\Lambda}^{2^{*}_s}
+O\Big(\frac{1}{\mu^{m+\theta}}\int_{K_2}||y|-\mu r_0|^{m+\theta}U_{\overline{x}_{1},\Lambda}^{2^{*}_s}\Big)\nonumber\\
&=-\frac{c_{0}}{\mu^m}\int_{\R^N}||y-\overline{x}_{1}|-\mu r_0|^m U_{0,\Lambda}^{2^{*}_s}
+O\Big(\frac{1}{\mu^{m+\theta}}\Big)\\
&=-\frac{c_{0}}{\mu^m}\int_{\R^N\backslash B_{\frac{\overline{x}_{1}}{2}}(0)}||y-\overline{x}_{1}|-\mu r_0|^m U_{0,\Lambda}^{2^{*}_s}
-\frac{c_{0}}{\mu^m}\int_{B_{\frac{\overline{x}_{1}}{2}}(0)}||y-\overline{x}_{1}|-\mu r_0|^m U_{0,\Lambda}^{2^{*}_s}\\
&\quad+\frac{1}{\mu^{m+\theta}}.
\end{align*}
However, we have
\begin{align*}
\frac{1}{\mu^m}\int_{\R^N\backslash B_{\frac{\overline{x}_{1}}{2}}(0)}||y-\overline{x}_{1}|-\mu r_0|^m U_{0,\Lambda}^{2^{*}_s}\leqslant C\int_{\R^N\backslash B_{\frac{\overline{x}_{1}}{2}}(0)}|\frac{|y-\overline{x}_{1}|}{\mu^m}+ r_0|^m U_{0,\Lambda}^{2^{*}_s}\leqslant\frac{C}{\mu^{N-\tau}}.
\end{align*}
If $y\in B_{\frac{\overline{x}_{1}}{2}}(0)$, since $|\overline{x}_{1}|=r$, by symmetry property, we have
$$\int_{B_{\frac{\overline{x}_{1}}{2}}(0)}||y-\overline{x}_{1}|-\mu r_0|^m U_{0,\Lambda}^{2^{*}_s}=
\int_{B_{\frac{\overline{x}_{1}}{2}}(0)}||y-re_1|^m U_{0,\Lambda}^{2^{*}_s},$$
where $e_1=(1,0,\cdots,0)$.
Thus, we obtain
\begin{align*}
&\quad\int_{B_{\frac{\overline{x}_{1}}{2}}(0)}||y-\overline{x}_{1}|-\mu r_0|^m U_{0,\Lambda}^{2^{*}_s}\\
&=\int_{B_{\frac{\overline{x}_{1}}{2}}(0)}|y_1|^m U_{0,\Lambda}^{2^{*}_s}+\frac{1}{2}m(m-1)\int_{B_{\frac{\overline{x}_{1}}{2}}(0)}|y_1|^{m-2} U_{0,\Lambda}^{2^{*}_s}|\mu r_0-r|^2+C|\mu r_0-r|^{2+\theta}.
\end{align*}
Hence, it holds that
\begin{align}\label{I21}
&\quad I_{21}=\int_{\Omega^{+}_1}K\Big(\frac{|y|}{\mu}\Big)U_{\overline{x}_{1},\Lambda}^{2^{*}_s}=\int_{\R^N}U^{2^{*}_s}
-\frac{c_{0}}{\mu^m\Lambda^m}\int_{B_{\frac{\overline{x}_{1}}{2}}(0)}|y_1|^m U^{2^{*}_s}
\nonumber\\&-\frac{c_{0}}{\mu^m\Lambda^{m-2}}\frac{1}{2}m(m-1)\int_{B_{\frac{\overline{x}_{1}}{2}}(0)}|y_1|^{m-2} U^{2^{*}_s}|\mu r_0-r|^2+\frac{C}{\mu^m}|\mu r_0-r|^{2+\theta}+\frac{C}{\mu^{m+\theta}}.
\end{align}
For $I_{22}$, we  show
\begin{align*}
&\quad\int_{\Omega_1^+}K\Big(\frac{|y|}{\mu}\Big)U_{\overline{x}_{1},\Lambda}^{2^{*}_s-1}(\sum_{j=2}^k U_{\overline{x}_{j},\Lambda}+\sum_{j=1}^kU_{\underline{x}_{j},\Lambda})\nonumber\\
&=\int_{\Omega_1^+}U_{\overline{x}_{1},\Lambda}^{2^{*}_s-1}(\sum_{j=2}^k U_{\overline{x}_{j},\Lambda}+\sum_{j=1}^kU_{\underline{x}_{j},\Lambda})+\int_{\Omega_1^+}\Big(K\Big(\frac{|y|}{\mu}\Big)-1\Big)U_{\overline{x}_{1},\Lambda}^{2^{*}_s-1}(\sum_{j=2}^k U_{\overline{x}_{j},\Lambda}+\sum_{j=1}^k U_{\underline{x}_{j},\Lambda})\\
&=\int_{\R^N}U_{\overline{x}_{1},\Lambda}^{2^{*}_s-1}(\sum_{j=2}^k U_{\overline{x}_{j},\Lambda}+\sum_{j=1}^kU_{\underline{x}_{j},\Lambda})+
\int_{\R^N\backslash\Omega_1^+}U_{\overline{x}_{1},\Lambda}^{2^{*}_s-1}(\sum_{j=2}^k U_{\overline{x}_{j},\Lambda}+\sum_{j=1}^kU_{\underline{x}_{j},\Lambda})\\
&\quad+\int_{\Omega_1^+}\Big(K\Big(\frac{|y|}{\mu}\Big)-1\Big)U_{\overline{x}_{1},\Lambda}^{2^{*}_s-1}(\sum_{j=2}^k U_{\overline{x}_{j},\Lambda}+\sum_{j=1}^k U_{\underline{x}_{j},\Lambda})\\
&:=\int_{\R^N}U_{\overline{x}_{1},\Lambda}^{2^{*}_s-1}(\sum_{j=2}^k U_{\overline{x}_{j},\Lambda}+\sum_{j=1}^kU_{\underline{x}_{j},\Lambda})+I_{221}+I_{222}.
\end{align*}
For $I_{221}$, we  calculate
\begin{align}
&\qquad\sum_{j=2}^k\int_{\R^N\backslash\Omega_1^+}U_{\overline{x}_{1},\Lambda}^{2^{*}_s-1} U_{\overline{x}_{j},\Lambda}\nonumber\\
&=\sum_{j=2}^k\int_{(\R^N\backslash\Omega_1^+)\bigcap B_{\overline{d}_{j}/2}(\overline{x}_{1})}U_{\overline{x}_{1},\Lambda}^{2^{*}_s-1} U_{\overline{x}_{j},\Lambda}+\sum_{j=2}^k\int_{(\R^N\backslash\Omega_1^+)\backslash B_{\overline{d}_{j}/2}(\overline{x}_{1})}U_{\overline{x}_{1},\Lambda}^{2^{*}_s-1} U_{\overline{x}_{j},\Lambda}\nonumber\\
&=\sum_{j=2}^k\int_{(\R^N\backslash\Omega_1^+)\bigcap B_{\overline{d}_{j}/2}(\overline{x}_{1})}U_{\overline{x}_{1},\Lambda}^{2^{*}_s-1} U_{\overline{x}_{j},\Lambda}+O\Big(\sum_{j=2}^k\frac{1}{|\overline{x}_{1}-\overline{x}_{j}|^{N-\epsilon}}\Big)\nonumber\\
&\leqslant\sum_{j=2}^k\int_{ B_{\overline{d}_{j}/2}(\overline{x}_{1})\backslash B_{\overline{d}_{2}/2}(\overline{x}_{1})}U_{\overline{x}_{1},\Lambda}^{2^{*}_s-1} U_{\overline{x}_{j},\Lambda}+O\Big(\sum_{j=2}^k\frac{1}{|\overline{x}_{1}-\overline{x}_{j}|^{N-\epsilon}}\Big)\nonumber\\
&\leqslant \sum_{j=2}^k\frac{1}{|\overline{x}_{1}-\overline{x}_{j}|^{N-2s}}
\int_{B_{\Lambda\overline{d}_{j}/2}(0)\backslash B_{\Lambda\overline{d}_{2}/2}(0)}\frac{1}{(1+z^2|)^{\frac{N+2s}{2}}}+O\Big(\sum_{j=2}^k\frac{1}{|\overline{x}_{1}-\overline{x}_{j}|^{N-\epsilon}}\Big)
\nonumber\\
&\leqslant \sum_{j=2}^k\frac{1}{|\overline{x}_{1}-\overline{x}_{j}|^{N-2s}}
O(\frac{1}{\overline{d}_{2}^{2s}})+O\Big(\sum_{j=2}^k\frac{1}{|\overline{x}_{1}-\overline{x}_{j}|^{N-\epsilon}}\Big),
\nonumber
\end{align}
where $\overline{d}_{j}=|\overline{x}_{1}-\overline{x}_{j}|$ for $j=2,\cdots,k$ and $\overline{d}_{2}=|\overline{x}_{1}-\overline{x}_{2}|=2r\sqrt{1-h^2}\sin\frac{\pi}{k}=O(\frac{r}{k}).$
However, we get
\begin{align*}
\sum_{j=1}^k\int_{\R^N\backslash\Omega_1^+}U_{\overline{x}_{1},\Lambda}^{2^{*}_s-1} U_{\underline{x}_{j},\Lambda}=O\Big(\frac{1}{(rh)^{N-\epsilon}}\frac{hk}{\sqrt{1-h^2}}\Big),
\nonumber
\end{align*}
Thus, it holds
\begin{align}\label{I221}
I_{221}=\int_{\R^N\backslash\Omega_1^+}U_{\overline{x}_{1},\Lambda}^{2^{*}_s-1}(\sum_{j=2}^k U_{\overline{x}_{j},\Lambda}+\sum_{j=1}^kU_{\underline{x}_{j},\Lambda})=O\Big(\frac{k^{N-\epsilon}}{(r\sqrt{1-h^2})^{N-\epsilon}}\Big)+O\Big(\frac{1}{(rh)^{N-\epsilon}}\frac{hk}{\sqrt{1-h^2}}\Big).
\end{align}

Finally, we compute $I_{222}$:
\begin{align*}
I_{222}&=\int_{\Omega_1^+\bigcap K_1}\Big(K\Big(\frac{|y|}{\mu}\Big)-1\Big)U_{\overline{x}_{1},\Lambda}^{2^{*}_s-1}(\sum_{j=2}^k U_{\overline{x}_{j},\Lambda}+\sum_{j=1}^k U_{\underline{x}_{j},\Lambda})\\
&\quad+\int_{\Omega_1^+\bigcap K_2}\Big(K\Big(\frac{|y|}{\mu}\Big)-1\Big)U_{\overline{x}_{1},\Lambda}^{2^{*}_s-1}(\sum_{j=2}^k U_{\overline{x}_{j},\Lambda}+\sum_{j=1}^k U_{\underline{x}_{j},\Lambda}).
\end{align*}
We consider the first term. For any $y\in\Omega_1^+$, we have
$$|y-\overline{x}_{j}|\geqslant|\overline{x}_{1}-\overline{x}_{j}|-|y-\overline{x}_{1}|\geqslant\frac{1}{4}|\overline{x}_{1}-\overline{x}_{j}|,\quad\text{if} \
|y-\overline{x}_{1}|\leqslant\frac{1}{4}|\overline{x}_{1}-\overline{x}_{j}|$$
and
$$|y-\overline{x}_{j}|\geqslant|y-\overline{x}_{1}|\geqslant\frac{1}{4}|\overline{x}_{1}-\overline{x}_{j}|,\quad\text{if} \
|y-\overline{x}_{1}|\geqslant\frac{1}{4}|\overline{x}_{1}-\overline{x}_{j}|.$$
Similarly, we have
$$|y-\underline{x}_{j}|\geqslant\frac{1}{4}|\overline{x}_{1}-\underline{x}_{j}|.$$
Thus, it holds that
\begin{align*}
(\sum_{j=2}^kU_{\overline{x}_{j},\Lambda}+\sum_{j=1}^kU_{\underline{x}_{j},\Lambda})&\leqslant\frac{C}{1+|y-\overline{x}_{1}|^{N-2s-\beta}}\Big(
\sum_{j=2}^k\frac{1}{1+|y-\overline{x}_{1}|^{\beta}}+\sum_{j=1}^k\frac{1}{1+|y-\underline{x}_{j}|^{\beta}}\Big)\nonumber\\
&\leqslant\frac{C}{1+|y-\overline{x}_{1}|^{N-2s-\beta}}\Big(
\sum_{j=2}^k\frac{1}{|\overline{x}_{j}-\overline{x}_{1}|^{\beta}}+\sum_{j=1}^k\frac{1}{|\overline{x}_{1}-\underline{x}_{j}|^{\beta}}\Big)\nonumber\\
&\leqslant\frac{C}{1+|y-\overline{x}_{1}|^{N-2s-\beta}}\Big(\frac{k^\beta}{(r\sqrt{1-h^2})^{\beta}}+\frac{1}{(rh)^\beta}\frac{hk}{\sqrt{1-h^2}}\Big),
\end{align*}
where $\beta\in(\frac{N-2s-m}{N-2s},\frac{N-2s}{2})$. So we  get
\begin{align*}
\int_{\Omega_1^+\bigcap K_1}\Big(K\Big(\frac{|y|}{\mu}\Big)-1\Big)U_{\overline{x}_{1},\Lambda}^{2^{*}_s-1}(\sum_{j=2}^k U_{\overline{x}_{j},\Lambda}+\sum_{j=1}^k U_{\underline{x}_{j},\Lambda})
&\leqslant\frac{C}{\mu^{N-\beta-\epsilon}}\Big(\frac{k^\beta}{(r\sqrt{1-h^2})^{\beta}}+\frac{1}{(rh)^\beta}\frac{hk}{\sqrt{1-h^2}}\Big)\\
&=O\Big(\frac{k^{N-\epsilon}}{(r\sqrt{1-h^2})^{N-\epsilon}}\Big)+O\Big(\frac{1}{(rh)^{N-\epsilon}}\frac{hk}{\sqrt{1-h^2}}\Big).
\end{align*}
For the second term, we have
\begin{align*}
\int_{\Omega_1^+\bigcap K_1}\Big(K\Big(\frac{|y|}{\mu}\Big)-1\Big)U_{\overline{x}_{1},\Lambda}^{2^{*}_s-1}(\sum_{j=2}^k U_{\overline{x}_{j},\Lambda}+\sum_{j=1}^k U_{\underline{x}_{j},\Lambda})
&=O\Big(\frac{k^{N-\epsilon}}{(r\sqrt{1-h^2})^{N-\epsilon}}\Big)+O\Big(\frac{1}{(rh)^{N-\epsilon}}\frac{hk}{\sqrt{1-h^2}}\Big).
\end{align*}
So, we obtain
\begin{align}\label{I222}
I_{222}=O\Big(\frac{k^{N-\epsilon}}{(r\sqrt{1-h^2})^{N-\epsilon}}\Big)+O\Big(\frac{1}{(rh)^{N-\epsilon}}\frac{hk}{\sqrt{1-h^2}}\Big).
\end{align}
At last, we  easily calculate $I_{23}$,
\begin{align}\label{I23}
I_{23}&=O\Big(\int_{\Omega^{+}_1}K\Big(\frac{|y|}{\mu}\Big)U_{\overline{x}_{1},\Lambda}^{2^*_s/2}(\sum_{j=2}^k U_{\overline{x}_{j},\Lambda}+\sum_{j=1}^kU_{\underline{x}_{j},\Lambda})^{2^*_s/2}\Big)\nonumber\\
&=O\Big(\frac{k^{N-\epsilon}}{(r\sqrt{1-h^2})^{N-\epsilon}}\Big)+O\Big(\frac{1}{(rh)^{N-\epsilon}}\frac{hk}{\sqrt{1-h^2}}\Big).
\end{align}
Combing \eqref{A4}, \eqref{I21}, \eqref{I221}, \eqref{I222} and \eqref{I23}, we  obtain
\begin{align*}
I(W_{r,h,\Lambda})&=k\Big((1-\frac{2}{2^*_s})\int_{\R^N}U^{2^{*}_s}+\frac{2c_0\int_{\R^N}|y_1|^mU^{2^{*}_s}}{2^{*}_s\Lambda^m \mu^m}+\frac{c_0m(m-1)\int_{\R^N}|y_1|^{m-2}U^{2^{*}_s}}{\Lambda^{m-2}\mu^m}(\mu r_{0}-r)^2\nonumber\\
&\quad-\frac{B_{1}k^{N-2s}}{\Lambda^{N-2s}(r\sqrt{1-h^2})^{N-2s}}
-\frac{B_2}{\Lambda^{N-2s}(rh)^{N-2s}}\frac{hk}{\sqrt{1-h^2}}+\frac{C}{\mu^m}|\mu r_0-r|^{2+\theta}+\frac{C}{\mu^{m+\theta}}\\
&\quad
+O\Big(\frac{1}{\mu^{N-2s-\frac{N-2s-m}{N-2s}-\frac{(N-2s-m)(N-2s-1)^2}{(N-2s)(N-2s+1)}+\theta}}\Big)\Big),
\end{align*}
where obviously,
$\frac{1}{\mu^{N-2s-\frac{N-2s-m}{N-2s}-\frac{(N-2s-m)(N-2s-1)^2}{(N-2s)(N-2s+1)}}}\leqslant\frac{1}{\mu^{m}}$.
\end{proof}
Meanwhile, we also make the the following expansions for $\frac{\partial I(W_{r,h,\Lambda})}{\partial\Lambda}$ and $\frac{\partial I(W_{r,h,\Lambda})}{\partial h}$.

\begin{Lem}\label{lemA6} We have
\begin{align*}
\frac{\partial I(W_{r,h,\Lambda})}{\partial\Lambda}&=k\Big(-\frac{mA_1}{\Lambda^{m+1} \mu^m}-\frac{A_2(m-2)}{\Lambda^{m-1}\mu^m}(\mu r_{0}-r)^2
+\frac{B_{1}(N-2s)k^{N-2s}}{\Lambda^{N-2s+1}(r\sqrt{1-h^2})^{N-2s}}\nonumber\\
&\quad+\frac{B_2(N-2s)}{\Lambda^{N-2s+1}(rh)^{N-2s}}\frac{hk}{\sqrt{1-h^2}}
+O\Big(\frac{1}{\mu^{N-2s-\frac{N-2s-m}{N-2s}-\frac{(N-2s-m)(N-2s-1)^2}{(N-2s)(N-2s+1)}+\theta}}\Big)\Big),
\end{align*}
where $  A_1, A_2, B_{1}\ and\ B_{2}$ are defined in Lemma \eqref{lemA5}.
\end{Lem}
\begin{proof}
The proof of this Lemma is similar to Lemma \ref{lemA5}, so we omit it here.
\end{proof}
\begin{Lem}\label{lemA7} We have
\begin{align*}
\frac{\partial I(W_{r,h,\Lambda})}{\partial h}&=k\Big(\frac{B_2(N-2s-1)k}{\Lambda^{N-2s}r^{N-2s}h^{N-2s}\sqrt{1-h^2}}
-\frac{B_1(N-2s)k^{N-2s}h}{\Lambda^{N-2s}r^{N-2s}(\sqrt{1-h^2})^{N-2s+2}}\\
&\quad+O\Big(\frac{1}{\mu^{N-2s-\frac{N-2s-m}{N-2s}-\frac{(N-2s-m)(N-2s-1)}{(N-2s)(N-2s+1)}+\theta}}
\Big)\Big),
\end{align*}
where $  A_1, A_2, B_{1}\ and\ B_{2}$ are defined in Lemma \eqref{lemA5}.
\end{Lem}
\begin{proof}
First,
\begin{align*}\frac{\partial I(W_{r,h,\Lambda})}{\partial h}&=\frac{1}{2}\frac{\partial }{\partial h}\Big(\int_{\R^N}|(-\Delta)^{\frac{s}{2}}W_{r,h,\Lambda}|^2-\frac{1}{2^*_s}\int_{\R^N}K\Big(\frac{|y|}{\mu}\Big)|W_{r,h,\Lambda}|^{2^{*}_s}\Big)\nonumber\\
&=k\frac{\partial }{\partial h}\int_{\R^N}U^{2^*_s-2}_{\overline{x}_{1},\Lambda}
(\sum_{j=2}^k U_{\overline{x}_{j},\Lambda}+\sum_{j=1}^kU_{\underline{x}_{j},\Lambda})
-\int_{\R^N}K\Big(\frac{|y|}{\mu}\Big)|W_{r,h,\Lambda}|^{2^{*}_s-1}\frac{\partial W_{r,h,\Lambda}}{\partial h}.
\end{align*}
Similar to the proof of Lemma \eqref{lemA5}, we  get,
\begin{align*}
\frac{\partial I(W_{r,h,\Lambda})}{\partial h}&=k\Big(\frac{B_2(N-2s-1)k}{\Lambda^{N-2s}r^{N-2s}h^{N-2s}\sqrt{1-h^2}}
-\frac{B_1(N-2s)k^{N-2s}h}{\Lambda^{N-2s}r^{N-2s}(\sqrt{1-h^2})^{N-2s+2}}\nonumber\\
&\quad-\frac{B_2k}{\Lambda^{N-2s}r^{N-2s}h^{N-2s-2}(\sqrt{1-h^2})^3}
+O\Big(\frac{1}{\mu^{N-2s-\frac{N-2s-m}{N-2s}-\frac{(N-2s-m)(N-2s-1)}{(N-2s)(N-2s+1)}+\theta}}\Big)\Big).
\end{align*}
In fact, we  know $\frac{B_2k}{\Lambda^{N-2s}r^{N-2s}h^{N-2s-2}(\sqrt{1-h^2})^3}
\leqslant O\Big(\frac{1}{\mu^{N-2s-\frac{N-2s-m}{N-2s}-\frac{(N-2s-m)(N-2s-1)}{(N-2s)(N-2s+1)}+\theta}}\Big)$.
\end{proof}

\bigskip
\noindent\textbf{Acknowledgements} \,\,\, Q. Guo was supported by NNSF of China (No. 11771469). The first author is grateful for  the useful discussion with Lipeng Duan.

\end{document}